\documentclass[reqno]{amsart}
\usepackage{cite}
\allowdisplaybreaks
\date{\today}

 \numberwithin{equation}{section}

\newcommand{\dv}{\mathrm{div}\,}

\newtheorem{Theorem}{Theorem}[section]
\newtheorem{Lemma}{Lemma}[section]

\newtheorem{Remark}{Remark}[section]

\begin{document}

\title[incompressible chemotaxis-Navier-Stokes system]
 {Uniform Regularity and Vanishing Viscosity\\
  limit for the chemotaxis-Navier-Stokes system \\
  in a 3D bounded domain}

 \author[Zhipeng Zhang]{Zhipeng Zhang}
\address{Department of Mathematics, Nanjing University, Nanjing
 210093, P.R. China}
 \email{zhpzhp@aliyun.com}

\begin{abstract}
We investigate the uniform regularity and vanishing viscosity limit for the incompressible chemotaxis-Navier-Stokes system in a smooth bounded domain $\Omega\subset\mathbb{R}^3$. It is shown that there exists a unique strong solution of the incompressible chemotaxis-Navier-Stokes system in a finite time interval which is independent of the viscosity coefficient.
Moreover, the solution is uniformly bounded in a conormal Sobolev space, which allows us to take the vanishing viscosity limit to
obtain the incompressible inviscid chemotaxis-Navier-Stokes system.
 \end{abstract}

\keywords{Incompressible chemotaxis-Navier-Stokes system, Conormal Sobolev space, Vanishing viscosity limit,
Navier boundary conditions}
\subjclass[2000]{35Q30, 76D03, 76D05, 76D07}
\maketitle
\section{Introduction}
Chemotaxis is a biological process in which cells or bacteria move towards a chemically more favorable environment. For example, bacteria move towards higher concentration of oxygen which they consume. A typical model describing chemotaxis is the Keller-Segel equations derived by Keller and Segel in \cite{KS} which have been studied extensively. In nature, bacteria often live in a viscous
fluid so that a convective transport of both cells and chemicals is happened through the fluid, and meanwhile a gravitation effect on the motion of the fluid is produced by the heavier bacteria. Thus, this interaction become more complicated since we not only pay attention to chemotaxis and diffusion but also transport and fluid dynamics. To describe the above biological phenomena, Tuval et al in \cite{TC} proposed the following model
\begin{align}
 &n_t+u\cdot\nabla n=\epsilon_1\Delta n-\nabla\cdot(k(c)n\nabla c),\label{1.1.1.1}\\
 &c_t+u\cdot\nabla c=\epsilon_2\Delta c-f(c) n ,\label{1.1.1.2}\\
 &u_t+u\cdot\nabla u+\nabla p=\epsilon_3 \Delta u-n\nabla\phi,\label{1.1.1.3}\\
 &\nabla\cdot u=0\label{1.1.1.4}
 \end{align}
in $(0,T)\times\Omega$. The unknowns in \eqref{1.1.1.1}--\eqref{1.1.1.4} are $n(t,x)$, $c(t,x)$, $u(t,x)$ and $p(t,x)$, denoting
the cell density, chemical concentration, velocity field and  pressure of the fluid, respectively. The pressure $p(t,x)$ in \eqref{1.1.1.3} can be recovered from $n$ and $u$ via an explicit Calder¨®n-Zygmund singular integral operator\cite{Ch}. The nonnegative functions $f(c)$ and $k(c)$ denote the chemical consumption rate and chemotaxis sensitivity. The given function $\phi$ represents the potential function produced by different physical mechanism, such as the gravitational force or centrifugal force. $\epsilon_i\,(i=1,2)$ are the corresponding diffusion coefficients for the cells and chemicals, and $\epsilon_3$ is the viscous coefficient for the fluid.

Due to the significance of the biological background, this model has been studied extensively and the main focus is on the solvability, see \cite{AL,FZ,JA,MKJ,RAP,MW,QX} and the references cited therein. Especially, Lorz \cite{AL} showed the local existence of weak solution for the above model in three bounded domain. Duan, Lorz and Markowich \cite{RAP} obtained the global existence of the solution of the system \eqref{1.1.1.1}-\eqref{1.1.1.4} and the time decay rates of the classical solution near constant states in $\mathbb{R}^3$. In \cite{MKJ}, Chae, Kang, and Lee proved the local well-posedness and blow up criterion of the smooth solution for the chemotaxis-Navier-Stokes system in $\mathbb{R}^d\, (d=2,3)$ and the global existence of the classical solution in $\mathbb{R}^2$ under some assumptions on the consumption rate and the chemotaxis sensitivity.

However, the research on the uniform regularity and vanishing viscosity limit for the system \eqref{1.1.1.1}-\eqref{1.1.1.4} is very limited. To the best of knowledge of the author, the only result is given by Zhang \cite{QZ1}. He proved the the inviscid limit of the 3D chemotaxis-Navier-Stokes system in the whole space and established the convergence rate.
From the biological point of review, it is more interesting to study this problem in a  bounded domain.

The purpose  of this paper is to investigate the uniform regularity and vanishing viscosity limit for the following  chemotaxis-Navier-Stokes system
\begin{align}
 &n^\epsilon_t+u^\epsilon\cdot\nabla n^\epsilon=\Delta n^\epsilon-\nabla\cdot(n^\epsilon\nabla c^\epsilon),\label{1.1.1}\\
 &c^\epsilon_t+u^\epsilon\cdot\nabla c^\epsilon=\Delta c^\epsilon-c^\epsilon n^\epsilon ,\label{1.1.2}\\
 &u^\epsilon_t+u^\epsilon\cdot\nabla u^\epsilon+\nabla p^\epsilon=\epsilon \Delta u^\epsilon-n^\epsilon\nabla\phi,\label{1.1.3}\\
 &\nabla\cdot u^\epsilon=0,\label{1.1.4}
 \end{align}
in $(0,T)\times\Omega$. Here, $\Omega$ is a smooth bounded domain of $\mathbb{R}^3$. The chemotaxis-Navier-Stokes system \eqref{1.1.1}-\eqref{1.1.4} is considered under the initial condition
\begin{align}\label{1.1.5}
(n^\epsilon,c^\epsilon,u^\epsilon)|_{t=0}=(n_0^\epsilon,c_0^\epsilon,u_0^\epsilon)
\end{align}
and the homogeneous boundary condition of Neumann type for $n^\epsilon$ and $c^\epsilon$
\begin{align}\label{1.2.1}
\frac{\partial n^\epsilon}{\partial\nu}=\frac{\partial c^\epsilon}{\partial\nu}=0,
\end{align}
where $\nu$ stands for the outward unit normal vector to $\Omega$,
and the Navier boundary condition for $u^\epsilon$ as
\begin{align}
&u^\epsilon\cdot \nu=0,\quad (Su^\epsilon\cdot \nu)_\tau=-\zeta u^\epsilon_\tau\quad\text{on}~~~~\partial\Omega,\label{1.2}
\end{align}
where $\zeta$ is a coefficient measuring the tendency of the fluid to slip on the boundary, $S$ is the strain tensor defined by
\begin{equation}
Su^\epsilon=\frac{1}{2}(\nabla u^\epsilon+(\nabla u^\epsilon)^t),\nonumber
\end{equation}
$(\nabla u^\epsilon)^t$ denotes the transpose of the matrix $\nabla u^\epsilon$, and $u_\tau^\epsilon$ stands for the tangential part of $u^\epsilon$ on $\partial\Omega$, i.e.
\begin{equation}
u^\epsilon_\tau=u^\epsilon-(u^\epsilon\cdot \nu)\nu.\nonumber
\end{equation}
 The boundary condition \eqref{1.2}
was introduced by Navier in \cite{NC} to show that the velocity is propositional to the tangential part of the stress.
It allow the fluid slip along the boundary and is often used to model rough boundaries.

We point out that when $n^\epsilon=c^\epsilon=0$ in the system \eqref{1.1.1}-\eqref{1.1.4}, it is reduced to the classical incompressible Navier-Stokes equations
\begin{align}
&u^\epsilon_t+u^\epsilon\cdot\nabla u^\epsilon+\nabla p^\epsilon=\epsilon\Delta u^\epsilon,\label{N1}\\
&\nabla\cdot u^\epsilon=0.\label{N2}
\end{align}
There are lots of results on the inviscid limit to the incompressible Navier-Stokes equations, see \cite{BS,HB,HB1,HC,IS,KT,Ma,MN,MR,OS,SH,SC1,SC2,XZ1} and the references therein. When the incompressible Navier-Stokes equations \eqref{N1}-\eqref{N2} are supplemented with the boundary condition
\begin{align}
 u^\epsilon\cdot \nu=0,\quad \nu\times\omega_u^\epsilon=0\quad\text{on}\quad\partial\Omega,
 \end{align}
 where $\omega_u^\epsilon=\nabla\times u^\epsilon$,
Xiao and Xin \cite{XZ1} obtained the local existence of strong solution with some uniform bounds in $H^3(\Omega)$ and the vanishing viscosity limit. Subsequently, their result was extended to $W^{k,p}(\Omega)$  in \cite{HB}. The main reason is that the boundary integrals vanishes on flat portions of the boundary, see also\cite{HB1,HC}.
Later, the results in \cite{XZ1,HB} were generalized by Berselli and Spirito  \cite{BS} to  a general bounded domain  under certain restrictions on the initial data.
Recently, Masmoudi and Rousset \cite{MR} considered the uniform regularity and vanishing viscosity limit for the incompressible Navier-Stokes equations \eqref{N1}-\eqref{N2}
with the Navier boundary condition \eqref{1.2} in the anisotropic conormal Sobolev spaces which will be defined below.

Motivated by the ideas of \cite{MR}, in this paper, we investigate the uniform regularity of the solution to the problem \eqref{1.1.1}-\eqref{1.2} in
the anisotropic conormal Sobolev spaces and take the inviscid limit $\epsilon\rightarrow0$ to obtain the following limit system (Assume that $(n^\epsilon,c^\epsilon,u^\epsilon)$ converge to $(n^0,c^0,u^0)$ in some sense.)
\begin{align}
 &n_t^0+u\cdot\nabla n^0=\Delta n^0-\nabla\cdot(n^0\nabla c^0),\label{1.7.1}\\
 &c_t^0+u\cdot\nabla c^0=\Delta c^0-c^0 n^0 ,\label{1.7.2}\\
 &u_t^0+u\cdot\nabla u^0+\nabla p^0=-n^0\nabla\phi,\label{1.7.3}\\
 &\nabla\cdot u^0=0,\label{1.7.4}
 \end{align}
in $(0,T)\times\Omega$ with the initial and boundary conditions
\begin{align}
&(n^0,c^0,u^0)|_{t=0}=(n_0,c_0,u_0),\label{1.15.1}\\
&u^0\cdot \nu=0,\quad \frac{\partial n^0}{\partial\nu}=\frac{\partial c^0}{\partial\nu}=0\quad \text{on}\quad \partial\Omega.\label{1.15}
\end{align}

Before stating our main results, we first introduce the notations and conventions used throughout this paper. We assume that $\Omega$  has a covering such that
\begin{equation}\label{c}
\Omega\subset\Omega_0 \cup_{k=1}^n\Omega_k,
 \end{equation}
where $\overline{\Omega_0}\subset\Omega$ and in each $\Omega_k$ there exists a function $\psi_k$ such that
\begin{align*}
\Omega\cup\Omega_k=\{\,x=(x_1,x_2,x_3)\,|\,x_3>\psi_k(x_1,x_2)\,\}\cup\Omega_k,\\
\partial\Omega\cup\Omega_k=\{\,x=(x_1,x_2,x_3)\,|\,x_3=\psi_k(x_1,x_2)\,\}\cup\Omega_k.
\end{align*}
We say that $\Omega$ is $\mathcal{C}^m$ if the functions $\psi_k$ are $\mathcal{C}^m$-functions.

To define the conormal Sobolev spaces, we consider $(Z_k)_{1\leq k\leq N}$, a finite set of generators of vector fields that are tangent to $\partial\Omega$, and set
\begin{align}
H^m_{co}(\Omega):=\big{\{} f\in L^2(\Omega)\,\big{|}\, Z^If\in L^2(\Omega)~~~ \text{for}~~\, |I|\leq m,\,\,\, m\in\mathbb{N}\big{\}},
\end{align}
where $I=(k_1,...,k_m)$, $Z^I:=Z_{k_1}\cdot\cdot \cdot Z_{k_m}$. We define the norm of $H^m_{co}(\Omega)$ as
$$ \|f\|^2_m:=\sum_{|I|\leq m}\|Z^If\|^2_{L^2}.$$
We say a vector field, $u$, is in $H^m_{co}(\Omega)$ if each of its components is in $H^m_{co}(\Omega)$ and
$$ \|u\|^2_m:=\sum_{i=1}^3\sum_{|I|\leq m}\|Z^Iu_i\|^2_{L^2}$$
is finite. In the same way, we set
$$ \|f\|_{m,\infty}:=\sum_{|I|\leq m}\|Z^If\|_{L^\infty},$$
$$\|\nabla Z^mf\|^2:=\sum_{|I|=m}\|\nabla Z^If\|^2_{L^2},$$
and we say that $f\in W^{m,\infty}_{co}(\Omega)$ if $\|f\|_{m,\infty}$ is finite.
By using the above covering of $\Omega$, we can assume that each vector field is supported in one of $\{\Omega_i\}_{i=0}^n$. Also, we note that the $\|\cdot\|_{m}$ norm yields a control of the standard $H^m$ norm in $\Omega_0$, whereas if $\Omega_i \cap \partial \Omega\neq {\emptyset}$, there is no control of the
normal derivatives.

Since $\partial\Omega$ is given locally by $x_3=\psi(x_1,x_2)$ (We omit the subscript $k$ for notational convenience), it is convenient to use the coordinates:
\begin{align}\label{1.18}
\Psi:(y,z)\mapsto(y,\psi(y)+z)=x.
\end{align}
A local basis is thus given by the vector fields $(\partial_{y^1},\partial_{y^1},\partial_z)$ where $\partial_{y^1}$ and $\partial_{y^2}$ are tangent to $\partial\Omega$ on the boundary
and in general $\partial_z$ is usually not a normal vector field. We sometimes use the notation $\partial_{y^3}$ for $\partial_z$. By using this parametrization, we can take suitable vector fields compactly supported in $\Omega_i$ in the definition of the $\|\cdot\|_m$ norms:
$$Z_i=\partial_{y^i}=\partial_i+\partial_i\psi\partial_z,~~i=1,2,\quad Z_3=\varphi(z)\partial_z,$$
where $\varphi(z)=\frac{z}{1+z}$ is a smooth and supported function in $[0,+\infty)$ and satisfies
$$\varphi(0)=0,~~\varphi'(0)>0,~~\varphi(z)>0\,\,\,\,\text{for}\,\,\,\,z>0.$$

In this paper, we shall still denote by $\partial_i, i=1,~2,~3$ or $\nabla$ the derivatives with respect to the standard coordinates of $\mathbb{R}^3$. The coordinates of a vector field $u$ in the basis $(\partial_{y^1},\partial_{y^1},\partial_z)$ will be denote by $u^i$, thus
$$u=u^1\partial_{y^1}+u^2\partial_{y^2}+u^3\partial_z.$$
We denote by $u_i$ the coordinates in the standard basis of $\mathbb{R}^3$, i.e.
$$u=u_1\partial_1+u_2\partial_2+u_3\partial_3.$$

The unit outward normal vector $\nu$ is given locally by
$$\nu(x)=\nu(\Psi(y,z)):=\frac{1}{\sqrt{1+|\nabla\psi(y)|^2}}
\left(
\begin{array}{c}
\partial_1\psi(y)\\
 \partial_2\psi(y)\\
 -1
\end{array}
\right),\quad N(x):=\sqrt{1+|\nabla\psi(y)|^2}\,\nu(x)$$
and denote by $\Pi$ the orthogonal projection
$$\Pi(x)u=\Pi(\Psi(y,z))u:=u-[u\cdot \nu(\Psi(y,z))]\nu(\Psi(y,z) $$
which gives the orthogonal projector onto the tangent space of the boundary. Note that both $\nu$ and $\Pi$ are defined in the whole $\Omega_k$ and do not depend on $z$.
By using these notations, the Navier boundary condition \eqref{1.2} reads
\begin{align}
u^\epsilon\cdot \nu=0, \quad \ \Pi\partial_\nu u^\epsilon=\theta(u^\epsilon)-2\zeta\Pi u^\epsilon,\label{1.12}
\end{align}
where $\theta$ is the shape operator (second fundamental form) of the boundary, $$\theta(u^\epsilon):=\Pi((\nabla \nu)u^\epsilon).$$

For later use and notational convenience, we set
\begin{align}
\mathcal{Z}^\alpha=\partial_t^{\alpha_0}Z^{\alpha_1}=\partial_t^{\alpha_0}Z^{\alpha_{11}}_1Z^{\alpha_{12}}_2Z^{\alpha_{13}}_3,
\end{align}
and we also use the following notations
\begin{align}
\|f(t)\|_{\mathcal{H}^m}^2:=\sum_{|\alpha|\leq m}\|\mathcal{Z}^\alpha f(t)\|^2_{L^2_x},\quad\|f(t)\|_{\mathcal{H}^{m,\infty}}^2:=\sum_{|\alpha|\leq m}\|\mathcal{Z}^\alpha f(t)\|^2_{L^\infty_x}
\end{align}
for smooth time-space function $f(t,x)$.

Throughout the paper, we shall denote by $\|\cdot\|_{H^m }$ and $\|\cdot\|_{W^{m,\infty} }$ the standard Sobolev norms
in $\Omega$ and the notation $|\cdot|_{H^m}$ will be used for the standard Sobolev norm of functions defined on $\partial\Omega$. Note that this norm involves only tangential derivatives. $\|\cdot\|$ stands for the standard $L^2$ norm and $(\cdot,\cdot)$ for the $L^2$ scalar product. The letter $D$ and $d$ are positive numbers which may change from line to line, but independent of $\epsilon\in(0,1]$. $D_m$ stands for a positive constant independent of $\epsilon$ which depends on the $\mathcal{C}^m$-norm of the functions $\psi_k$. $P(\cdot)$ denotes a polynomial function.

In order to obtain the uniform estimates for the solutions of the chemotaxis-Navier-Stokes system with the boundary conditions \eqref{1.2.1} and \eqref{1.2}, we need to find a suitable functional space. Here, we define the functional space $\mathcal{E}_m^\epsilon(T)$ for functions $(n^\epsilon,c^\epsilon,u^\epsilon)(t,x)$ as follows:
\begin{align}\label{1.14}
\mathcal{E}_m^\epsilon(T)=\big{\{}(n^\epsilon,c^\epsilon,u^\epsilon)\in L^\infty([0,T],L^2)\, \big{|}\,\, \mbox{esssup}_{0\leq t\leq T}\|(n^\epsilon,c^\epsilon,u^\epsilon)(t)\|_{\mathcal{E}_m^\epsilon}<+\infty\big{\}},
\end{align}
where the norms $\|(\cdot,\cdot,\cdot)\|_{\mathcal{E}_m^\epsilon}$ is given by
\begin{align}
\|(n^\epsilon,c^\epsilon,u^\epsilon)(t)\|_{\mathcal{E}_m^\epsilon}:=&\|(n^\epsilon,c^\epsilon,u^\epsilon)\|_{\mathcal{H}^m}+\|(\nabla n^\epsilon,\nabla u^\epsilon)(t)\|_{\mathcal{H}^{m-1}}
+\|\nabla c^\epsilon\|_{\mathcal{H}^{m}}\nonumber\\
&+\|(\Delta n^\epsilon,\Delta c^\epsilon)\|_{\mathcal{H}^{m-1}}+\|\nabla u^\epsilon\|_{\mathcal{H}^{1,\infty}}.
\end{align}
Correspondingly, for the initial data $(n_0^\epsilon,c_0^\epsilon,u_0^\epsilon)$, we difine
\begin{align}\label{1.17}
\sup_{0<\epsilon\leq 1}\|(n^\epsilon_0,c^\epsilon_0,u^\epsilon_0)\|_{\mathcal{E}_m^\epsilon}:=\sup_{0<\epsilon\leq 1}&\big{\{}\|(n^\epsilon_0,c^\epsilon_0,u^\epsilon_0)\|_{\mathcal{H}^m}+\|(\nabla n^\epsilon_0,\nabla u^\epsilon_0)\|_{\mathcal{H}^{m-1}}
+\|\nabla c^\epsilon_0\|_{\mathcal{H}^{m}}\nonumber\\
&+\|(\Delta n^\epsilon_0,\Delta c^\epsilon_0)\|_{\mathcal{H}^{m-1}}+\|\nabla u^\epsilon_0\|_{\mathcal{H}^{1,\infty}}\big{\}}\leq \widetilde{D}_0,
\end{align}
where $\widetilde{D}_0$ is a positive constant independent of $\epsilon\in(0,1]$, and the time derivatives of initial data in \eqref{1.17} are defined through the system \eqref{1.1.1}-\eqref{1.1.4}. Thus, the initial data $(n^\epsilon_0,c^\epsilon_0,u^\epsilon_0)$ is assumed to have a higher space regularity and compatibility. We note that the a priori estimates in Theorem \ref{Th3.1} below are obtained in the situation that the approximate solution is sufficiently smooth up to the boundary. Therefore, in order to obtain a self-contained result, we need to assume the approximated initial data satisfies the boundary compatibility condition \eqref{1.2}. For the initial data $(n^\epsilon_0,c^\epsilon_0,u^\epsilon_0)$ satisfying  \eqref{1.17}, it is not clear if there exists an approximate sequence $(n^{\epsilon,\delta}_0,c^{\epsilon,\delta}_0,u^{\epsilon,\delta}_0)$ ($\delta$ being a regularization parameter), which satisfy the boundary compatibilities and $\|(n^{\epsilon,\delta}_0-n^{\epsilon}_0,c^{\epsilon,\delta}_0-c^{\epsilon}_0,u^{\epsilon,\delta}_0-u^{\epsilon}_0)\|_{\mathcal{E}_m^\epsilon}\rightarrow0$ as $\delta\rightarrow0$. Thus, we set
\begin{align}\label{1.26}
\mathcal{E}^{m}_{CNS,ap}:=\big{\{}&(n^\epsilon,c^\epsilon,u^\epsilon)\in \mathcal{C}^{2{m+1}}\times \mathcal{C}^{2{m+1}}\times \mathcal{C}^{2m}\, \big{|}\,\,\nonumber\\
 &\partial_t^kn^\epsilon,\partial_t^kc^\epsilon,\partial_t^ku^\epsilon,\partial_t^k\nabla n^\epsilon,\partial_t^k\nabla c^\epsilon, k=1,...,m  \,\,\text{are defined through} \nonumber\\
 &\text{the system} \,\eqref{1.1.1}-\eqref{1.2}\, \text{and }  \partial_t^ku^\epsilon, \partial_t^k\nabla n^\epsilon,\partial_t^k\nabla c^\epsilon,k=1,...,m-1 \nonumber\\
 &\text{satisfy the boundary compatibility condition}\big{\}}
\end{align}
and
\begin{align}
\mathcal{E}^{m}_{CNS}:= \text{The closure of} \,\,\mathcal{E}^{m,\epsilon}_{CNS,ap} \,\,\text{in the norm} \,\,\|(\cdot,\cdot,\cdot)\|_{\mathcal{E}_m^\epsilon}.
\end{align}

Our first result of this paper reads as follows:

\begin{Theorem}\label{Th1}
 Let $m$ be an integer satisfying $m\geq6$ and $\Omega$ be a $\mathcal{C}^{m+2}$ domain. Assume that the initial data $(n^\epsilon_0,c^\epsilon_0,u^\epsilon_0)\in\mathcal{E}^{m}_{CNS}$ satisfy \eqref{1.17}, $\nabla\cdot u^\epsilon_0=0$ and $\phi$ is a smooth function. Then, there exist $\widetilde{T}_0>0$ and $\widetilde{D}_1$, independent of $\epsilon\in(0,1]$ and $|\zeta|\leq1$, such that there exists a unique solution of the problem \eqref{1.1.1}-\eqref{1.2} on $[0,\widetilde{T}_0]$ that satisfies
\begin{align}\label{1.28}
&\sup_{0\leq\tau\leq t}(\|(n^{\epsilon},c^{\epsilon},u^{\epsilon})\|_{\mathcal{H}^m}^2+\|\nabla(n^{\epsilon},u^{\epsilon})\|_{\mathcal{H}^{m-1}}^2+\|\nabla c^{\epsilon}\|_{\mathcal{H}^{m}}^2\nonumber\\
&+\|\Delta(n^{\epsilon},c^{\epsilon})\|_{\mathcal{H}^{m-1}}^2+\|\nabla u^{\epsilon}\|^2_{1,\infty})
+\epsilon\int_0^t(\|\nabla u^{\epsilon}\|^2_{\mathcal{H}^{m}}+\|\nabla^2 u^{\epsilon}\|^2_{\mathcal{H}^{m-1}})\,d\tau\nonumber\\
&+\int_0^t(\|\nabla n^{\epsilon}\|^2_{\mathcal{H}^{m}}+\|\Delta c^{\epsilon}\|^2_{\mathcal{H}^{m}}
+\|\nabla\Delta(n^{\epsilon},c^{\epsilon})\|^2_{\mathcal{H}^{m-1}})\,d\tau\leq\widetilde{D}_1,\quad \forall\, t\in[0,\widetilde{T}_0],
\end{align}
where $\widetilde{D}_1$ depends only on $\widetilde{D}_0$, $D_{m+2}$ and $\phi$.
 \end{Theorem}
 \begin{Remark}\label{R1.1}
Since $\Delta c^\epsilon$ appear in the equation \eqref{1.1.1}, we can not do the same higher order estimate for $\nabla n^\epsilon$ as $\nabla c^\epsilon$. Otherwise, we obtain the term $\| \nabla\Delta c^\epsilon\|_{\mathcal{H}^{m}}$ in the right-hand side of the energy inequality, which is out of control.
\end{Remark}
\begin{Remark}
 When the Navier boundary condition \eqref{1.2} is replaced by the following more generalized form
 \begin{equation}\label{GNB1}
 u^\epsilon\cdot \nu=0,\quad \nu\times\omega_u^\epsilon=[Bu^\epsilon]_\tau\quad\text{on}\quad\partial\Omega,
 \end{equation}
where $B=2(A-S(\nu))$ and A is a $(1,1)$-type tensor on the boundary $\partial \Omega$, we can still obtain the same results as those in Theorem \ref{Th1}.
 \end{Remark}
Now we give some comments on the proof of Theorem \ref{Th1}. We shall follow and modify some ideas developed in \cite{MR}. In fact, due to the strong coupling among $n^\epsilon$, $c^\epsilon$ and $u^\epsilon$, we need to overcome some new difficulties and to face more complicated energy estimates. The main step of the proof Theorem \ref{Th1} can be stated as follows:
First, we get a conormal energy estimates in $\mathcal{H}^m$ for $(n^\epsilon,c^\epsilon,u^\epsilon)$.
The second step is to give the estimates for $\| \nabla n^\epsilon\|_{\mathcal{H}^{m-1}}$ and $\| \nabla c^\epsilon\|_{\mathcal{H}^{m}}$. Since the equation \eqref{1.1.1} involve the term $\Delta c^\epsilon$, we can not do the same higher order estimate for $\nabla n^\epsilon$ as $\nabla c^\epsilon$.
In the third step, we show the estimates for $\|\Delta(n^\epsilon,c^\epsilon)\|_{\mathcal{H}^{m-1}}$. Since the dissipative terms $\Delta n^\epsilon$ and $\Delta c^\epsilon$ appear, we can easily deduce the estimates.
Next, we focus on the estimate of $\|\nabla u^\epsilon\|_{\mathcal{H}^{m-1}}$ in the fourth step. In order to obtain this estimate by an energy method, $\partial_\nu u^\epsilon$ is not a convenient quantity, since it does not vanish on the boundary. However, due to the incompressible condition \eqref{1.1.4}, $\|\partial_\nu u^\epsilon\cdot \nu\|_{\mathcal{H}^{m-1}}$ can be easily controlled by the $\mathcal{H}^{m}$ norm of $u^\epsilon$. Moreover, thank to the the Nvier boundary conditions \eqref{1.2}, it is convenient to study $\eta^\epsilon=(Su^\epsilon \nu+\zeta u^\epsilon)_\tau$. We find that $\eta^\epsilon$ satisfy equations with homogeneous Dirichlet boundary conditions and gives us control of $(\partial_\nu u^\epsilon)_\tau$, and by performing energy estimate on the equations solved by $\eta^\epsilon$, we can get a control of $\|\eta^\epsilon\|_{\mathcal{H}^{m-1}}$.
The fifth step is to estimate the pressure. In the spirit of \cite{MR}, we split the pressure into two parts which satisfy nonhomogeneous elliptic equations with Neumann boundary conditions. In view of the regularity theory of elliptic equations with Neumann boundary conditions, we get the estimates of the pressure terms.
Finally, we need to estimate $\|u^\epsilon\|_{\mathcal{H}^{1,\infty}}$, $\|\nabla (c^\epsilon,u^\epsilon)\|_{\mathcal{H}^{1,\infty}}$, $\|(n^\epsilon,c^\epsilon)\|_{W^{2,\infty}}$ and $\|\nabla\Delta c^\epsilon\|_{L^\infty}$. By virtue of the anistropic Sobolev embedding inequality in Lemma \ref{L2.2} and the equations \eqref{1.1.1} and \eqref{1.1.2}, we can give the estimates for them, except for $\|\nabla u^\epsilon\|_{\mathcal{H}^{1,\infty}}$.
In order to estimate $\|\nabla u^\epsilon\|_{\mathcal{H}^{1,\infty}}$, similar to the fourth step, we find equivalent quantities $\widetilde{\eta}^\epsilon$ which satisfies a homogenous Dirichlet condition and solves a convection-diffusion equation. The estimate will be obtained by using Lemma 14 in \cite{MR}.

Based on Theorem \ref{Th1}, we justify the vanishing viscosity limit as follows:
\begin{Theorem}\label{Th2}
 Let $m$ be an integer satisfying $m\geq6$ and $\Omega$ be a $\mathcal{C}^{m+2}$ domain. Consider $(n_0,c_0,u_0)\in\mathcal{E}^{m}_{CNS}$ and $\nabla\cdot u_0=0$ and $(n^{\epsilon},c^{\epsilon},u^{\epsilon})$ the solution of the system \eqref{1.1.1}-\eqref{1.1.4} with the initial data $(n_0,c_0,u_0)$ and the boundary conditions \eqref{1.2.1} and \eqref{1.2} given by Theorem \ref{Th1}. Then, there exists a unique solution of the system \eqref{1.7.1}-\eqref{1.15} with initial value $(n_0,c_0,u_0)$, $(n^{0},c^{0},u^{0})\in\mathcal{E}^{m}_{CNS}$ such that
 \begin{align}
 \sup_{0\leq t\leq\widetilde{T}_0}\big{(}&\|u^\epsilon-u^0\|_{L^2}+\|(n^\epsilon,c^\epsilon)-(n^0,c^0)\|_{H^1}\nonumber\\
 &+\|u^\epsilon-u^0\|_{L^\infty}+\|(n^\epsilon,c^\epsilon)-(n^0,c^0)\|_{W^{1,\infty}}\big{)}\rightarrow0,
 \end{align}
when $\epsilon$ tends to zero.
\end{Theorem}

The rest of the paper is organized as follows. In the following section, we present some inequalities that will be used frequently later.
In Section \ref{Sec3}, we prove a priori energy estimates. Next, we use the a priori estimates in Theorem \ref{Th3.1} to give the proof of Theorem \ref{Th1} in section \ref{Sec4}.
Finally, we prove Theorem \ref{Th2} in Section \ref{Sec5}.


\section{Preliminaries}\label{Sec2}

We first introduce the Korn's inequlity which play an important role in energy estimates below.
\begin{Lemma} [Korn's inequality\cite{FP}]\label{L2.1} Let
$\Omega$ be a bounded Lipschitz domain of $\mathbb{R}^3$. There exists a constant $D>0$ depending only on $\Omega$ such that
\begin{equation}\nonumber
\|u\|_{H^1(\Omega)}\leq D\,(\|u\|_{L^2(\Omega)}+\|S(u)\|_{L^2(\Omega)}),\quad \forall~ u\in(H^1(\Omega))^3.
\end{equation}
\end{Lemma}
Next, we introduce the space
\begin{align}
\mathcal{W}^m([0,T]\times\Omega)=\{f(t,x)\in L^2([0,T]\times\Omega)| \mathcal{Z}^\alpha f\in L^2([0,T]\times\Omega),|\alpha|\leq m\}.
\end{align}
Then, we have the following Gagliardo-Nirenberg-Moser type inequality whose proof can be found in \cite{GO}.
\begin{Lemma}\label{L2.3}Let $u, v\in L^\infty([0,T]\times\Omega)\cap \mathcal{W}^m([0,T]\times\Omega)$, we have
\begin{align}
\int^t_0\|\mathcal{Z}^{\alpha_1}u\mathcal{Z}^{\alpha_2}v\|^2d\tau\leq D\, (\|u\|_{L^\infty_{t,x}}^2\int^t_0\|v\|_{\mathcal{H}^m}^2d\tau+\|v\|_{L^\infty_{t,x}}^2&\int^t_0\|u\|_{\mathcal{H}^m}^2d\tau),\nonumber\\
 &|\alpha_1|+|\alpha_2|=m.\nonumber
\end{align}
\end{Lemma}
Finally, we need the following anistropic Sobolev embedding and trace estimates.
\begin{Lemma}[\!\!\cite{MR,WXY}]\label{L2.2} Let $m_1\geq0$ and $m_2\geq0$ be integers, $u\in H^{m_1}_{co}(\Omega)\cap H^{m_2}_{co}(\Omega)$ and $\nabla u\in H^{m_2}_{co}(\Omega)$.
Then we have
\begin{align*}
&\|u\|^2_{{L^\infty(\Omega)}}\leq D\,(\|\nabla u\|_{m_2}+\|u\|_{m_2})\|u\|_{m_1},\quad m_1+m_2\geq 3,\\
&|u|^2_{H^s(\partial\Omega)}\leq D\,(\|\nabla u\|_{m_2}+\|u\|_{m_2})\|u\|_{m_1},\quad m_1+m_2\geq 2s\geq 0.
\end{align*}
\end{Lemma}

\section{A priori estimates}\label{Sec3}

The main aim of this section is to prove the following a priori estimate which is the crucial step in the proof of Theorem \ref{Th1}. For notational convenience, we drop the superscript $\epsilon$
throughout this section.

\begin{Theorem}\label{Th3.1}
For $m>6$ and a $\mathcal{C}^{m+2}$ domain $\Omega$, there exists a constant $D_{m+2}>0$, independent of $\epsilon\in(0,1]$ and $|\zeta|\leq1$, such that for any sufficiently smooth solution defined on $[0,T]$ of the problem \eqref{1.1.1}-\eqref{1.2} in $\Omega$, we have
\begin{align}\label{3.1}
&\sup_{0\leq\tau\leq t}N_m(\tau)+\epsilon\int_0^t(\|\nabla u\|^2_{\mathcal{H}^{m}}+\|\nabla^2 u\|^2_{\mathcal{H}^{m-1}})\,d\tau+\int_0^t(\|\nabla n\|^2_{\mathcal{H}^{m}}+\|\Delta c\|^2_{\mathcal{H}^{m}}\nonumber\\
&+\|\nabla\Delta(n,c)\|^2_{\mathcal{H}^{m-1}})\,d\tau\leq\,\widetilde{D}_2\,D_{m+2}\,\Big{\{}N_m(0)
+\big{(}1+P(N_m(t))\big{)}\int_0^tP(N_m(\tau))d\tau\Big{\}},
\end{align}
where $\widetilde{D}_2$ depends only on $\phi$ and
\begin{align}\label{N_m}
\!\!\! N_m(t):=\|(n,c,u)\|_{\mathcal{H}^m}^2+\|\nabla(n,u)\|_{\mathcal{H}^{m-1}}^2+\|\nabla c\|_{\mathcal{H}^{m}}^2+\|\Delta(n,c)\|_{\mathcal{H}^{m-1}}^2+\|\nabla u\|^2_{1,\infty}.
\end{align}
\end{Theorem}

Since the proof of Theorem \ref{Th3.1} is quite complicated and lengthy, we divided the proof into the following subsections.
\subsection{Conormal Energy Estimates} In this subsection, we first give the basic $L^2$ energy estimates.
\begin{Lemma}\label{L3.1} For a smooth solution of the problem \eqref{1.1.1}-\eqref{1.2}, we have
\begin{align}\label{3.3.1}
\sup_{0\leq\tau\leq t}\|u\|^2+\epsilon\int_0^t\| \nabla u\|^2d\tau\leq\,D_2\,\Big{(}\|u_0\|^2+\int_0^t(\|u\|^2+\|n\nabla\phi\|^2)\,d\tau\Big{)}
\end{align}
for every $\epsilon\in(0,1]$ and $|\zeta|\leq 1$.
\end{Lemma}

\begin{proof}
Multiplying \eqref{1.1.3} by $u$, we obtain
\begin{align}
( u_t,u)+( u\cdot\nabla u,u)=\epsilon(\Delta u,u)-(\nabla p,u)-(n\nabla\phi,u).\label{2.2}
\end{align}
Due to \eqref{1.1.4}, \eqref{1.2} and integration by parts, we find
\begin{align}
(\nabla p,u)&=\int_{\partial\Omega}p u\cdot \nu \,d\sigma-\int_{\Omega}p\nabla\cdot u \,dx=0,\label{3.5}\\
(\epsilon\Delta u,u)=2\epsilon(\nabla\cdot Su,u)&=-2\epsilon\|Su\|^2+2\epsilon\ \int_{\partial \Omega}((Su)\cdot \nu)\cdot u\,d\sigma \nonumber\\
&=-2\epsilon\| Su\|^2-2\epsilon\zeta\int_{\partial\Omega}|u_\tau|^2\,d\sigma.\label{3.6}
\end{align}
From \eqref{2.2}-\eqref{3.6} and Lemma \ref{L2.1}, we can obtain \eqref{3.3.1}.
\end{proof}

Next, we give the basic $L^2$ energy estimate for $(n,c)$.
\begin{Lemma}\label{L3.2} For a smooth solution of the problem \eqref{1.1.1}-\eqref{1.2}, we have
\begin{align}\label{3.3.1.1}
\sup_{0\leq\tau\leq t}\|n\|^2+\int_0^t\| \nabla n\|^2d\tau\leq\,D_2\,\Big{(}\|n_0\|^2+\int_0^t\|n\nabla c\|^2\,d\tau\Big{)}
\end{align}
for every $\epsilon\in(0,1]$ and $|\zeta|\leq 1$.
\end{Lemma}

\begin{proof}
Multiplying \eqref{1.1.1} by $n$, we obtain
\begin{align}
( n_t,n)+( u\cdot\nabla n,n)=(\Delta n,n)-(\nabla\cdot(n\nabla c),n).\label{3.8}
\end{align}
Due to \eqref{1.1.4}-\eqref{1.2} and integration by parts, we find
\begin{align}
( u\cdot\nabla n,n)&=0,\quad(\Delta n,n)=-\|\nabla n\|^2,\quad(\nabla\cdot(n\nabla c),n)=-(n\nabla c,\nabla n).\label{3.9.1}
\end{align}
Based on \eqref{3.9.1} and Young's inequality, we can obtain \eqref{3.3.1.1}.
\end{proof}
Similar to Lemma \ref{L3.2}, we can easily get the following Lemma, the proof being omitted.
\begin{Lemma}\label{L3.3} For a smooth solution of the problem \eqref{1.1.1}-\eqref{1.2}, we have
\begin{align}\label{3.3.1.2}
\sup_{0\leq\tau\leq t}\|c\|^2+\int_0^t\| \nabla c\|^2d\tau\leq\,D_2\,\Big{(}\|c_0\|^2+\int_0^t\|\sqrt{n}c\|^2\,d\tau\Big{)}
\end{align}
for every $\epsilon\in(0,1]$ and $|\zeta|\leq 1$.
\end{Lemma}

Now, we turn to the higher order energy estimates of $(n,c,u)$. Set
\begin{align}
M(t):=\sup_{0\leq\tau\leq t}\big{\{}\|u\|^2_{\mathcal{H}^{1,\infty}}+\|\nabla (c,u)\|^2_{\mathcal{H}^{1,\infty}}+\|(n,c)\|^2_{W^{2,\infty}}+\|\nabla\Delta c\|_{L^\infty}^2\big{\}}.
\end{align}
\begin{Lemma}\label{L3.4} For every $m\geq 0$, a smooth solution of the problem \eqref{1.1.1}-\eqref{1.2} satisfies the estimate
\begin{align}\label{3.8.1}
&\sup_{0\leq\tau\leq t}\|u\|^2_{\mathcal{H}^m}+\epsilon\int_0^t\|\nabla u\|^2_{\mathcal{H}^{m}}d\tau\nonumber\\
\leq&\,D_{m+2}\,\Big{\{}\|u_0\|^2_{\mathcal{H}^m}+\int_0^t\big{(}\|\nabla^2p_1\|_{\mathcal{H}^{m-1}}\|u\|_{\mathcal{H}^{m}}
+\epsilon^{-1}(\|\nabla p_2\|_{\mathcal{H}^{m-1}}^2+\|p_2\|_{\mathcal{H}^{m-1}}^2)\big{)}d\tau\nonumber\\
&+\big{(}1+P(M(t))\big{)}\int_0^t\big{(}\|u\|_{\mathcal{H}^m}^2+\|\nabla u\|_{\mathcal{H}^{m-1}}^2\big{)}d\tau+\int_0^t\|n\nabla\phi\|_{\mathcal{H}^{m}}^2d\tau\Big{\}},
\end{align}
where the pressure $p:=p_1+p_2$. Here, $p_1$ is the``Euler" part of the pressure which solves
\begin{equation}\label{3.8.11}
\left\{\begin{array}{l}
\Delta p_1=-\nabla\cdot( u\cdot\nabla u)-\nabla\cdot(n\nabla\phi)\quad \text{in}\quad \Omega,\\
\partial_\nu p_1=-(u\cdot\nabla u)\cdot \nu-n\nabla\phi\cdot\nu\quad \text{on}\quad \partial\Omega
\end{array}
\right.
\end{equation}
and $p^\epsilon_2$ is the ``Navier-Stokes" part of the pressure which solves
\begin{equation}\label{3.9}
\left\{\begin{array}{l}
\Delta p_2=0\quad \text{in}\quad \Omega,\\
\partial_\nu p_2=\epsilon\Delta u\cdot \nu\quad \text{on}\quad \partial\Omega.
\end{array}
\right.
\end{equation}
Note that here we use the convention that $\|\cdot\|_{\mathcal{H}^{m}}=0$ for $m<0$.
\end{Lemma}

\begin{proof}
The estimate for $m=0$ has been given in Lemma \ref{L3.1}. Now we assume that Lemma \ref{L3.2} holds for $|\alpha|\leq m-1$ and prove that it is still ture for $|\alpha|=m$. We apply $\mathcal{Z}^\alpha$ to \eqref{1.1.3} for $|\alpha|=m$ to obtain
\begin{align}
&\mathcal{Z}^\alpha u_t+ u\cdot\nabla \mathcal{Z}^\alpha u+\mathcal{Z}^\alpha\nabla p=\epsilon\mathcal{Z}^\alpha\Delta u-\mathcal{Z}^\alpha(n\nabla\phi)+\mathcal{C}_1,\label{3.11}
\end{align}
where
\begin{align*}
\mathcal{C}_1:=-[\mathcal{Z}^\alpha,u\cdot\nabla]u=-\sum_{|\beta|\geq1,\beta+\gamma=\alpha}D_{\beta,\gamma}\mathcal{Z}^\beta( u)\cdot\mathcal{Z}^\gamma\nabla u- u\cdot[\mathcal{Z}^\alpha,\nabla]u.
\end{align*}
Multiplying \eqref{3.11} by $\mathcal{Z}^\alpha u$ and integrating by parts, we have
\begin{align}\label{3.12}
 \frac{1}{2}\frac{d}{dt}\|\mathcal{Z}^\alpha u\|^2+(\mathcal{Z}^\alpha \nabla p,\mathcal{Z}^\alpha u)
=\epsilon(\mathcal{Z}^\alpha\Delta u,\mathcal{Z}^\alpha u)-(\mathcal{Z}^\alpha(n\nabla\phi),\mathcal{Z}^\alpha u)+(\mathcal{C}_1,\mathcal{Z}^\alpha u).
\end{align}

First, we estimate the first term in the right-hand side of \eqref{3.12}. We obtain
\begin{align}\label{3.13}
\epsilon\int_\Omega \mathcal{Z}^\alpha\Delta u\cdot \mathcal{Z}^\alpha u \,dx=2\epsilon\int_\Omega(\nabla\cdot \mathcal{Z}^\alpha Su)\cdot \mathcal{Z}^\alpha u\,dx
+2\epsilon\int_\Omega([\mathcal{Z}^\alpha,\nabla\cdot]Su)\cdot\mathcal{Z}^\alpha u\,dx.
\end{align}
Now, by integrating by parts, we get from the first term in the right-hand side of \eqref{3.13} that
\begin{align}\label{3.14}
\epsilon\int_\Omega(\nabla\cdot \mathcal{Z}^\alpha Su)\cdot \mathcal{Z}^\alpha u\,dx=&-\epsilon\int_\Omega \mathcal{Z}^\alpha Su\cdot\nabla \mathcal{Z}^\alpha u\,dx
+\epsilon\int_{\partial\Omega}((\mathcal{Z}^\alpha Su)\cdot \nu)\cdot \mathcal{Z}^\alpha u\,d\sigma\nonumber\\
=&-\epsilon\|S(\mathcal{Z}^\alpha u)\|^2-\epsilon\int_\Omega[\mathcal{Z}^\alpha,S]u\cdot \nabla \mathcal{Z}^\alpha u\,dx\nonumber\\
&+\epsilon\int_{\partial\Omega}((\mathcal{Z}^\alpha Su)\cdot \nu)\cdot \mathcal{Z}^\alpha u\,d\sigma.
\end{align}
Thanks to Lemma \ref{L2.1}, there exists a $d_0>0$ such that
\begin{align}\label{3.15}
\epsilon\int_\Omega(\nabla\cdot \mathcal{Z}^\alpha Su)\cdot& \mathcal{Z}^\alpha u\,dx
\leq-d_0\epsilon\|\nabla(\mathcal{Z}^\alpha u)\|^2+D_{m+2}\|u\|^2_{\mathcal{H}^m}\nonumber\\
&+D_{m+2}\epsilon\|\nabla \mathcal{Z}^\alpha u\|\|\nabla u\|_{\mathcal{H}^{m-1}}+\epsilon\int_{\partial\Omega}((\mathcal{Z}^\alpha Su)\cdot \nu)\cdot \mathcal{Z}^\alpha u\,d\sigma.
\end{align}
It remains to estimate the boundary term of \eqref{3.15}. Before we treat the boundary term, we have the following observations. Due to the Navier boundary condition \eqref{1.12}, we get
\begin{align}\label{3.16}
|\Pi\partial_\nu\partial_t^{\alpha_0}u|_{H^m(\partial\Omega)}\leq|\theta(\partial_t^{\alpha_0}u)|_{H^m(\partial\Omega)}+2\zeta|\Pi \partial_t^{\alpha_0}u|_{H^m(\partial\Omega)}\leq D\,|\partial_t^{\alpha_0}u|_{H^m(\partial\Omega)}.
\end{align}
To estimate the normal part of $\partial_\nu u$, we can use the divergence free condition to write
\begin{equation}\label{3.17}
\nabla\cdot \partial_t^{\alpha_0}u=\partial_\nu\partial_t^{\alpha_0}u\cdot \nu+(\Pi\partial_{y_1}\partial_t^{\alpha_0}u)^1+(\Pi\partial_{y_2}\partial_t^{\alpha_0}u)^2.
\end{equation}
Hence, we easily get
\begin{equation}\label{3.18}
|\partial_\nu\partial_t^{\alpha_0}u\cdot \nu|_{H^{m-1}(\partial\Omega)}\leq D\,|\partial_t^{\alpha_0}u|_{H^{m}(\partial\Omega)}.
\end{equation}
From \eqref{3.16} and \eqref{3.18}, we have
\begin{equation}\label{3.19}
|\nabla \partial_t^{\alpha_0}u|_{H^{m-1}(\partial\Omega)}\leq D\,|\partial_t^{\alpha_0}u|_{H^{m}(\partial\Omega)}.
\end{equation}
Thanks to $u\cdot n=0$ on the boundary, we immediately obtain that
\begin{equation}\label{3.20}
|(\mathcal{Z}^\alpha u)\cdot n|_{H^1(\partial\Omega)}\leq D\,|\partial_t^{\alpha_0}u|_{H^{m-|\alpha_0|}(\partial\Omega)},\quad|\alpha|=m.
\end{equation}
Now we return to deal with the boundary term of \eqref{3.15} as follows
\begin{align*}
\epsilon\int_{\partial\Omega}((\mathcal{Z}^\alpha Su)\cdot \nu)\cdot \mathcal{Z}^\alpha u\,d\sigma=&\epsilon\int_{\partial\Omega}\mathcal{Z}^\alpha(\Pi(Su\cdot \nu))\cdot\Pi \mathcal{Z}^\alpha u\,d\sigma\nonumber\\
&+\epsilon\int_{\partial\Omega}\mathcal{Z}^\alpha(\partial_\nu u\cdot \nu)\mathcal{Z}^\alpha u\cdot \nu\,d\sigma+\mathcal{C}_b,
\end{align*}
where
\begin{align*}
\mathcal{C}_b:=&\,\epsilon\int_{\partial\Omega}[\mathcal{Z}^\alpha,\Pi](Su\cdot \nu)\cdot\Pi \mathcal{Z}^\alpha u\,d\sigma+\epsilon\int_{\partial\Omega}[\mathcal{Z}^\alpha,\nu](Su^\epsilon\cdot \nu)\mathcal{Z}^\alpha u\cdot \nu\,d\sigma\\
&+\epsilon\sum_{\beta+\gamma=\alpha,\gamma\neq0}D_{\beta,\gamma}\int_{\partial\Omega}(\mathcal{Z}^\beta Su\cdot\mathcal{Z}^\gamma \nu)\cdot\mathcal{Z}^\alpha u\,d\sigma.
\end{align*}
Thanks to the Navier boundary condition \eqref{1.2}, we can easily get
\begin{align}
&\Big{|}\epsilon\int_{\partial\Omega}\mathcal{Z}^\alpha(\Pi(Su\cdot \nu))\cdot\Pi \mathcal{Z}^\alpha u\,d\sigma\Big{|}\leq \epsilon \, D_{m+2}\,|\partial_t^{\alpha_0}u|_{H^{m-|\alpha_0|}(\partial\Omega)}^2.\label{3.23}
\end{align}
We also note that $\mathcal{Z}^\alpha u\cdot n$ and $\mathcal{C}_b$ with $\alpha=\alpha_0$ vanish, so we assume $\alpha\neq\alpha_0$.
From \eqref{3.19}, we obtain that
\begin{align}
|C_b|&\leq \epsilon\,D_{m+2}\,|\nabla\partial_t^{\alpha_0} u|_{H^{m-1-|\alpha_0|}(\partial\Omega)}|\partial_t^{\alpha_0}u|_{H^{m-|\alpha_0|}(\partial\Omega)}\nonumber\\
&\leq \epsilon \, D_{m+2}\,|\partial_t^{\alpha_0}u|_{H^{m-|\alpha_0|}(\partial\Omega)}^2.\label{3.22}
\end{align}
By integrating by parts along the boundary, we have that
\begin{align}\label{3.24}
\Big{|}\epsilon\int_{\partial\Omega}\mathcal{Z}^\alpha(\partial_\nu u\cdot \nu)\mathcal{Z}^\alpha u\cdot \nu\,d\sigma\Big{|}\leq&\,\epsilon\, D\,|\partial_\nu\partial_t^{\alpha_0}u\cdot \nu|_{H^{m-1-|\alpha_0|}(\partial\Omega)}|\mathcal{Z}^\alpha u\cdot \nu|_{H^1(\partial\Omega)}\nonumber\\
\leq&\,\epsilon\, D\,|\partial_t^{\alpha_0}u|^2_{H^{m-|\alpha_0|}(\partial\Omega)}.
\end{align}
Hence, we get from \eqref{3.15} and \eqref{3.23}-\eqref{3.24} that
\begin{align}\label{3.25}
\epsilon\int_\Omega(\nabla\cdot \mathcal{Z}^\alpha Su)\cdot \mathcal{Z}^\alpha u\,d\sigma
\leq\, D_{m+2}\,&(\|u\|^2_{\mathcal{H}^m}+\epsilon\|\nabla \mathcal{Z}^\alpha u\|\|\nabla u\|_{\mathcal{H}^{m-1}}\nonumber\\
&+\epsilon|\partial_t^{\alpha_0}u|^2_{H^{m-|\alpha_0|}(\partial\Omega)})-d_0\epsilon\|\nabla(\mathcal{Z}^\alpha u)\|^2.
\end{align}
Next, we deal with the second term of the right-hand side of \eqref{3.13}, i.e.\linebreak
$\epsilon\int_\Omega([\mathcal{Z}^\alpha,\nabla\cdot]Su)\cdot \mathcal{Z}^\alpha u\,dx$. We can expand it as a sum of terms under the form
\begin{equation}
\epsilon\int_\Omega\beta_k\partial_k(\mathcal{Z}^{\tilde{\alpha}}Su)\cdot \mathcal{Z}^\alpha u\,dx,\quad |\tilde{\alpha}|\leq m-1.\nonumber
\end{equation}
By using integrations by parts and \eqref{3.19}, we have
\begin{align}\label{3.27.1}
\epsilon\Big{|}\int_\Omega\beta_k\partial_k(\mathcal{Z}^{\tilde{\alpha}}Su)\cdot \mathcal{Z}^\alpha u\,dx\Big{|}\leq\, D_{m+2}\,\epsilon&\big{(}\|\nabla u\|_{\mathcal{H}^m}\|\nabla u\|_{\mathcal{H}^{m-1}}
+\|u\|^2_{\mathcal{H}^m}\nonumber\\
&+|\partial_t^{\alpha_0}u|_{H^{m-|\alpha_0|}(\partial\Omega)}^2\big{)}.
\end{align}
Consequently, from \eqref{3.25} and \eqref{3.27.1}, we get
\begin{align}\label{3.28}
\epsilon\Big{|}\int_\Omega \mathcal{Z}^\alpha \Delta u\cdot \mathcal{Z}^\alpha u\,dx\Big{|}\leq\, D_{m+2}\,&\big{\{}\|u\|^2_{\mathcal{H}^{m}}+\epsilon\|\nabla u\|_{\mathcal{H}^{m}}\|\nabla u\|_{\mathcal{H}^{m-1}}\nonumber\\
&+\epsilon|\partial_t^{\alpha_0}u|^2_{H^{m-|\alpha_0|}(\partial\Omega)}\big{\}}-d_0\epsilon\|\nabla(\mathcal{Z}^\alpha u)\|^2.
\end{align}

Second, we estimate the term involving the pressure $p$ in \eqref{3.12}. We note that $(\mathcal{Z}^\alpha \nabla p,\mathcal{Z}^\alpha u)$ with $\alpha=\alpha_0$ vanishes, so we deal with the case of $\alpha\neq\alpha_0$.
\begin{align}\label{3.26}
\Big{|}\int_{\Omega}\mathcal{Z}^\alpha \nabla p\cdot \mathcal{Z}^\alpha u\,dx\Big{|}\leq&\, \|\nabla^2p_1\|_{\mathcal{H}^{m-1}}\|u\|_{\mathcal{H}^{m}}+\Big{|}\int_{\Omega} \mathcal{Z}^\alpha \nabla p_2\cdot  \mathcal{Z}^\alpha u\,dx\Big{|}.
\end{align}
Now, we focus on the last term in \eqref{3.26}. By integrating by parts, we obtain that
\begin{align}\label{3.27}
\Big{|}\int_{\Omega}\mathcal{Z}^\alpha \nabla p_2\cdot\mathcal{Z}^\alpha u\,dx\Big{|}\leq D_{m+2}\,&\Big{(}\|\nabla p_2\|_{\mathcal{H}^{m-1}}\|u\|_{\mathcal{H}^{m}}+\|\nabla p_2\|_{\mathcal{H}^{m-1}}\|\nabla\mathcal{Z}^\alpha u\|\nonumber\\
&+\Big{|}\int_{\partial\Omega}\mathcal{Z}^\alpha  p_2\mathcal{Z}^\alpha u\cdot \nu\,d\sigma\Big{|}\Big{)}.
\end{align}
By integrating by parts along the boundary and Lemma \ref{L2.2}, we get
\begin{align}\label{3.29}
&\Big{|}\int_{\partial\Omega}\mathcal{Z}^\alpha  p_2\mathcal{Z}^\alpha u\cdot \nu\,d\sigma\Big{|}\nonumber\\
\leq&\, D_{m+2}\,|\mathcal{Z}^{\widetilde{\alpha}}p_2|_{L^2(\partial\Omega)}|\mathcal{Z}^\alpha u\cdot \nu|_{H^1(\partial\Omega)}\nonumber\\
\leq&\,D_{m+2}\,(\|\nabla p_2\|_{\mathcal{H}^{m-1}}+\|p_2\|_{\mathcal{H}^{m-1}})(\|\nabla u\|_{\mathcal{H}^{m}}+\|u\|_{\mathcal{H}^{m}}),
\end{align}
where $|\widetilde{\alpha}|=m-1$. From \eqref{3.26}-\eqref{3.29}, we get
\begin{align}\label{3.33}
\Big{|}\int_{\Omega}\mathcal{Z}^\alpha \nabla p\cdot \mathcal{Z}^\alpha u\,dx\Big{|}\leq&\,D_{m+2}\,\|\nabla^2p_1\|_{\mathcal{H}^{m-1}}\|u\|_{\mathcal{H}^{m}}+D_{m+2}\,(\|\nabla p_2\|_{\mathcal{H}^{m-1}}\nonumber\\
&+\|p_2\|_{\mathcal{H}^{m-1}})(\|\nabla u\|_{\mathcal{H}^{m}}+\|u\|_{\mathcal{H}^{m}}).
\end{align}

Finally, we estimate the commutator term. By using Lemma \ref{L2.3}, we have
\begin{align}
\int_0^t\|\mathcal{C}_1\|^2d\tau\leq&\,\sum_{|\beta|\geq1,\beta+\gamma=\alpha}D_{\beta,\gamma}\int_0^t\|\mathcal{Z}^\beta( u)\cdot\mathcal{Z}^\gamma\nabla u\|^2\,d\tau+\int_0^t\|u\cdot[\mathcal{Z}^\alpha,\nabla]u\|^2\,d\tau\nonumber\\
\leq& \,D\Big{\{}\|\mathcal{Z}u\|^2_{L^\infty_{x,t}}\int_0^t\|\nabla u\|^2_{\mathcal{H}^{m-1}}d\tau+\|\nabla u\|^2_{L^\infty_{x,t}}\int_0^t\| \mathcal{Z}u\|^2_{\mathcal{H}^{m-1}}d\tau\nonumber\\
&\quad+\| u\|^2_{L^\infty_{x,t}}\int_0^t\| \nabla u\|^2_{\mathcal{H}^{m-1}}d\tau\Big{\}}\nonumber\\
\leq& \,D_{m+1}M(t)\int_0^t(\|\nabla u\|^2_{\mathcal{H}^{m-1}}+\|u\|^2_{\mathcal{H}^{m}})d\tau,\label{3.31}
\end{align}

Consequently, from \eqref{3.28}, \eqref{3.33}-\eqref{3.31}, Lemma \ref{L2.2}, Young's inequality and the assumptions with respect to $|\alpha|\leq m-1$, we get \eqref{3.8.1}.
This ends the proof of Lemma \ref{L3.2}.
\end{proof}
Next, we give the higher order estimate of $(n,c)$.
\begin{Lemma}\label{L3.5} For every $m\geq 0$, a smooth solution of the problem \eqref{1.1.1}-\eqref{1.2} satisfies the estimate
\begin{align}\label{3.8.1.1}
&\sup_{0\leq\tau\leq t}(\|(n,c)\|^2_{\mathcal{H}^m})+d\int_0^t\|\nabla(n,c)\|^2_{\mathcal{H}^{m}}d\tau\nonumber\\
\leq&\,D_{m+2}\,\Big{\{}\|(n_0,c_0\|^2_{\mathcal{H}^m}+\big{(}1+P(M(t))\big{)}\int_0^t\big{(}\|(n,c,u)\|_{\mathcal{H}^m}^2+\|\nabla c \|_{\mathcal{H}^{m}}^2\nonumber\\
&\quad\quad\quad+\|\nabla n \|_{\mathcal{H}^{m-1}}^2+\|\nabla^2 c \|_{\mathcal{H}^{m-1}}^2\big{)}d\tau+\delta\int_0^t\|\nabla^2n\|_{\mathcal{H}^{m-1}}^2d\tau\Big{\}},
\end{align}
where $\delta>0$ is a small enough constant.
\end{Lemma}
\begin{proof}
The estimate for $m=0$ has been given in Lemmas \ref{L3.2} and \ref{L3.3}. Now we assume that Lemma \ref{L3.5} holds for $|\alpha|\leq m-1$ and prove that it is still true for $|\alpha|=m$. We apply $\mathcal{Z}^\alpha$ to \eqref{1.1.1} for $|\alpha|=m$ to obtain that
\begin{align}
&\mathcal{Z}^\alpha n_t+ u\cdot\nabla \mathcal{Z}^\alpha n=\mathcal{Z}^\alpha\Delta n-\mathcal{Z}^\alpha(\nabla\cdot (n\nabla c))+\mathcal{C}_2,\label{3.37}
\end{align}
where
\begin{align*}
\mathcal{C}_2:=-[\mathcal{Z}^\alpha,u\cdot\nabla]n=-\sum_{|\beta|\geq1,\beta+\gamma=\alpha}D_{\beta,\gamma}\mathcal{Z}^\beta( u)\cdot\mathcal{Z}^\gamma\nabla n- u\cdot[\mathcal{Z}^\alpha,\nabla]n.
\end{align*}
Multiplying \eqref{3.37} by $\mathcal{Z}^\alpha n$ and integrating by parts, we have
\begin{align}\label{3.38}
 \frac{1}{2}\frac{d}{dt}\|\mathcal{Z}^\alpha n\|^2
=(\mathcal{Z}^\alpha\Delta n,\mathcal{Z}^\alpha n)-(\mathcal{Z}^\alpha(\nabla\cdot (n\nabla c)),\mathcal{Z}^\alpha n)+(\mathcal{C}_2,\mathcal{Z}^\alpha n).
\end{align}

First, we estimate the first term in the right-hand side of \eqref{3.38}. By integrating by parts, we get
\begin{align}\label{3.39}
\int_\Omega \mathcal{Z}^\alpha\Delta n \mathcal{Z}^\alpha n\,dx=&\int_\Omega\nabla\cdot(\mathcal{Z}^\alpha\nabla n)\mathcal{Z}^\alpha n\,dx
+\int_\Omega[\mathcal{Z}^\alpha,\nabla\cdot]\nabla n\mathcal{Z}^\alpha n\,dx\nonumber\\
=&-\int_\Omega\mathcal{Z}^\alpha\nabla n\nabla\mathcal{Z}^\alpha n\,dx+\int_\Omega[\mathcal{Z}^\alpha,\nabla\cdot]\nabla n\mathcal{Z}^\alpha n\,dx\nonumber\\
&+\int_{\partial\Omega}\nu\cdot\mathcal{Z}^\alpha\nabla n\mathcal{Z}^\alpha n\,d\sigma.
\end{align}
Thanks to Lemma \ref{L2.2}, there exists a constant $d>0$ such that
\begin{align}\label{3.40}
\int_\Omega \mathcal{Z}^\alpha\Delta n \mathcal{Z}^\alpha n\,dx\leq-d\|\nabla n\|^2_{\mathcal{H}^m}+D_\delta D_{m+2}(\| n\|^2_{\mathcal{H}^m}+\|\nabla n\|^2_{\mathcal{H}^{m-1}})+\delta\|\nabla^2 n\|^2_{\mathcal{H}^{m-1}}.
\end{align}

Next, we deal with the second term in the right-hand side of \eqref{3.38}. By integrating by parts, we get
\begin{align}\label{3.41}
-\int_\Omega\mathcal{Z}^\alpha(\nabla\cdot(n\nabla c))\mathcal{Z}^\alpha n\,dx=&\int_\Omega\mathcal{Z}^\alpha(n\nabla c)\nabla\mathcal{Z}^\alpha n\,dx-\int_\Omega[\mathcal{Z}^\alpha,\nabla\cdot](n\nabla c)\mathcal{Z}^\alpha n\,dx\nonumber\\
&-\int_{\partial\Omega}\nu\cdot\mathcal{Z}^\alpha(n\nabla c)\mathcal{Z}^\alpha n\,d\sigma.
\end{align}
For the last term of \eqref{3.41}, by using Lemmas \ref{L2.3} and \ref{L2.2}, and \eqref{1.2.1}, we obtain that
\begin{align}\label{3.42}
\Big{|}\int_0^t\int_{\partial\Omega}\nu\cdot\mathcal{Z}^\alpha(n\nabla c)\mathcal{Z}^\alpha n\,d\sigma d\tau\Big{|}\leq& \,D_\delta D_{m+2}(1+M(t))\int^t_0\big{\{}\| n\|^2_{\mathcal{H}^m}+\|\nabla(n,c)\|^2_{\mathcal{H}^{m-1}}\nonumber\\
&+\|\nabla^2 c\|^2_{\mathcal{H}^{m-1}}\big{\}}\,d\tau+\delta\int^t_0\|\nabla n\|^2_{\mathcal{H}^{m}}\,d\tau.
\end{align}
For the commutator term in \eqref{3.41},  we can expand it as a sum of terms under the form
\begin{align}
\int_\Omega\beta_k\partial_k\mathcal{Z}^{\tilde{\alpha}}(n\nabla c) \mathcal{Z}^\alpha n\,dx,\quad |\tilde{\alpha}|\leq m-1.\nonumber
\end{align}
By using Lemma \ref{L2.3}, we easily get that
\begin{align}\label{3.43}
\Big{|}\int^t_0\int_\Omega[\mathcal{Z}^\alpha,\nabla\cdot](n\nabla c)\mathcal{Z}^\alpha n\,dxd\tau\Big{|}\leq\,D_{m+1} M(t)\int^t_0&(\| n\|^2_{\mathcal{H}^m}+\|\nabla(n,c)\|^2_{\mathcal{H}^{m-1}}\nonumber\\
&+\|\nabla^2 c\|^2_{\mathcal{H}^{m-1}})d\tau.
\end{align}
Also, by using Lemma \ref{L2.3}, we obtain that
\begin{align}\label{3.44}
\Big{|}\int^t_0\int_\Omega\mathcal{Z}^\alpha(n\nabla c)\nabla\mathcal{Z}^\alpha n\,dxd\tau\Big{|}\leq\,&D_\delta D_{m+2}M(t)\int^t_0(\| n\|^2_{\mathcal{H}^m}+\|\nabla c\|^2_{\mathcal{H}^{m}})d\tau\nonumber\\
&+\delta\int^t_0\|\nabla n\|^2_{\mathcal{H}^{m}}d\tau.
\end{align}
Therefore, from \eqref{3.42}-\eqref{3.44}, we get that
\begin{align}\label{3.45}
&\Big{|}\int_0^t\int_\Omega\mathcal{Z}^\alpha(\nabla\cdot(n\nabla c))\mathcal{Z}^\alpha n\,dx d\tau\Big{|}\nonumber\\
\leq &\,D_\delta D_{m+2}(1+M(t))\int^t_0\big{(}\| n\|^2_{\mathcal{H}^m}+\|\nabla n\|^2_{\mathcal{H}^{m-1}}\nonumber\\
&+\|\nabla c\|^2_{\mathcal{H}^{m}}+\|\nabla^2 c\|^2_{\mathcal{H}^{m-1}}\big{)}d\tau+\delta\int^t_0\|\nabla n\|^2_{\mathcal{H}^{m}}d\tau.
\end{align}

Finally, we estimate the commutator term in \eqref{3.38}. Similar to \eqref{3.31}, by using Lemma \ref{L2.3}, we have
\begin{align}\label{3.46}
\Big{|}\int_0^t\int_\Omega\mathcal{C}_2\mathcal{Z}^\alpha n\,dxd\tau\Big{|}
\leq& \,D_{m+1}M(t)\int_0^t\big{(}\|(n,u)\|^2_{\mathcal{H}^{m}}+\|\nabla n\|^2_{\mathcal{H}^{m-1}}\big{)}d\tau.
\end{align}

Similar to $n$, we apply $\mathcal{Z}^\alpha$ to \eqref{1.1.2} for $|\alpha|=m$ to obtain
\begin{align}
&\mathcal{Z}^\alpha c_t+ u\cdot\nabla \mathcal{Z}^\alpha c=\mathcal{Z}^\alpha\Delta c-\mathcal{Z}^\alpha(nc))+\mathcal{C}_3,\label{3.47}
\end{align}
where
\begin{align*}
\mathcal{C}_3:=-[\mathcal{Z}^\alpha,u\cdot\nabla]c.
\end{align*}
Multiplying \eqref{3.47} by $\mathcal{Z}^\alpha c$ and integrating by parts, we have
\begin{align}\label{3.48}
 \frac{1}{2}\frac{d}{dt}\|\mathcal{Z}^\alpha c\|^2
=(\mathcal{Z}^\alpha\Delta c,\mathcal{Z}^\alpha c)-(\mathcal{Z}^\alpha(nc),\mathcal{Z}^\alpha c)+(\mathcal{C}_3,\mathcal{Z}^\alpha c).
\end{align}
By using Lemma \ref{L2.3}, we can easily obtain
\begin{align}\label{3.49}
 \Big{|}\int^t_0\int_\Omega\mathcal{Z}^\alpha(nc)\mathcal{Z}^\alpha c\,dxd\tau\Big{|}\leq D_{m+2}M(t)\int^t_0(\| n\|^2_{\mathcal{H}^m}+\| c\|^2_{\mathcal{H}^m})\,d\tau.
\end{align}
Similar to \eqref{3.40} and \eqref{3.46}, we can directly get
\begin{align}
&\int_\Omega \mathcal{Z}^\alpha\Delta c \mathcal{Z}^\alpha c\,dx\leq-d\|\nabla c\|^2_{\mathcal{H}^m}+D_\delta D_{m+2}(\| c\|^2_{\mathcal{H}^m}+\|\nabla c\|^2_{\mathcal{H}^{m-1}})+\delta\|\nabla^2 c\|^2_{\mathcal{H}^{m-1}},\label{3.50}\\
&\Big{|}\int_0^t\int_\Omega\mathcal{C}_3\mathcal{Z}^\alpha c\,dxd\tau\Big{|}
\leq \,D_{m+1}M(t)\int_0^t(\|(c,u)\|^2_{\mathcal{H}^{m}}+\|\nabla c\|^2_{\mathcal{H}^{m-1}})d\tau.\label{3.51}
\end{align}

Consequently, from \eqref{3.40}, \eqref{3.45}-\eqref{3.46}, \eqref{3.49}-\eqref{3.51}, and the assumptions with respect to $|\alpha|\leq m-1$, we get \eqref{3.8.1.1}.
This ends the proof of Lemma \ref{L3.5}.
\end{proof}

\subsection{Conormal Energy Estimates for $\nabla n$ and $\nabla c$}
In this subsection, we shall give some uniform estimates to $\|\nabla n\|_{\mathcal{H}^{m-1}}$ and $\|\nabla c\|_{\mathcal{H}^{m}}$. First, we deal with $\nabla n$. We have

\begin{Lemma}\label{L3.6} For every $m\geq 1$, a smooth solution of the problem \eqref{1.1.1}-\eqref{1.2} satisfies the estimate
\begin{align}\label{3.52}
&\sup_{0\leq\tau\leq t}\|\nabla n\|^2_{\mathcal{H}^{m-1}}+\int_0^t\|\Delta n\|^2_{\mathcal{H}^{m-1}}d\tau\nonumber\\
\leq&\,D_{m+2}\,\Big{\{}\|\nabla n_0\|^2_{\mathcal{H}^{m-1}}+\delta\int^t_0\|\nabla\Delta n\|^2_{\mathcal{H}^{m-2}}d\tau+\delta\int^t_0\|\nabla\Delta c\|^2_{\mathcal{H}^{m-1}}d\tau\nonumber\\
&+D_\delta\big{(}1+P(M(t))\big{)}\int_0^t\big{(}\|(n,u)\|_{\mathcal{H}^{m-1}}^2+\|\nabla(n,c,u)\|_{\mathcal{H}^{m-1}}^2\nonumber\\
&+\|\nabla^2 (n,c) \|_{\mathcal{H}^{m-1}}^2\big{)}d\tau\Big{\}},
\end{align}
where $\delta>0$ is a small enough constant.
\end{Lemma}

\begin{proof}
Multiplying \eqref{1.1.1} by $\Delta n$ yields that
\begin{align}
\int_\Omega n_t\Delta n\,dx+\int_\Omega u\cdot\nabla n\Delta n\,dx=\int_\Omega\Delta n\Delta n\,dx-\int_\Omega\nabla\cdot(n\nabla c)\Delta n\,dx.
\end{align}
Integrating by parts and using the boundary condition \eqref{1.2.1}, we obtain that
\begin{align}
\int_\Omega n_t\Delta n\,dx=-\frac{1}{2}\frac{d}{dt}\int_\Omega |\nabla n|^2\,dx.
\end{align}
Furthermore, by using Young's inequality, we get
\begin{align}
\sup_{0\leq\tau\leq t}\|\nabla n\|^2+\int_0^t\|\Delta n\|^2\,d\tau\leq D\Big{\{}\|\nabla n_0\|^2+M(t)\int_0^t(\|\nabla n\|^2+\| n\|^2)\,d\tau\Big{\}}.\nonumber
\end{align}
Hence, the case of $|\alpha|=1$ is true.

Next, we consider the higher order estimates. Assume that \eqref{3.52} has been proved for $|\alpha|\leq m-2$, we need to prove it holds for $|\alpha|=m-1$. By applying $\mathcal{Z}^\alpha\nabla$ with $|\alpha|=m-1$ to \eqref{1.1.1}, we obtain that
\begin{align}\label{3.56}
 \mathcal{Z}^\alpha\nabla n_t-\mathcal{Z}^\alpha\nabla\Delta n=-\mathcal{Z}^\alpha\nabla(\nabla n\cdot\nabla c)-\mathcal{Z}^\alpha\nabla(n\Delta c)-\mathcal{Z}^\alpha\nabla(u\cdot\nabla n).
\end{align}
Multiplying \eqref{3.56} by $\mathcal{Z}^\alpha\nabla n$ leads to
\begin{align*}
&\frac{1}{2}\frac{d}{dt}\int_\Omega|\mathcal{Z}^\alpha\nabla n|^2\,dx-\int_\Omega\mathcal{Z}^\alpha\nabla \Delta n\mathcal{Z}^\alpha\nabla  n\,dx\nonumber\\
=&-\int_\Omega\mathcal{Z}^\alpha\nabla(\nabla n\cdot\nabla c)\mathcal{Z}^\alpha\nabla n\,dx
-\int_\Omega\mathcal{Z}^\alpha\nabla(n\Delta c)\mathcal{Z}^\alpha\nabla n\,dx\nonumber\\
&-\int_\Omega\mathcal{Z}^\alpha\nabla(u\cdot\nabla n)\mathcal{Z}^\alpha\nabla n\,dx.
\end{align*}
By using integration by parts and the boundary condition \eqref{1.2.1}, we obtain that
\begin{align*}
&-\int_\Omega\mathcal{Z}^\alpha\nabla \Delta n\mathcal{Z}^\alpha\nabla  n\,dx\nonumber\\
=&-\int_{\partial\Omega}\mathcal{Z}^\alpha \Delta n\mathcal{Z}^\alpha \nabla n\cdot\nu\,d\sigma+\int_{\Omega}|\dv(\mathcal{Z}^\alpha \nabla n)|^2\,dx\nonumber\\
&+\int_{\Omega}[\mathcal{Z}^\alpha,\dv]\nabla n\dv(\mathcal{Z}^\alpha \nabla n)\,dx-\int_{\Omega}[\mathcal{Z}^\alpha,\nabla]\Delta n\cdot\mathcal{Z}^\alpha \nabla n\,dx.
\end{align*}
Hence, we have
\begin{align}\label{3.57}
&\frac{1}{2}\frac{d}{dt}\int_\Omega|\mathcal{Z}^\alpha\nabla n|^2\,dx+\int_{\Omega}|\dv(\mathcal{Z}^\alpha \nabla n)|^2\,dx\nonumber\\
=&-\int_\Omega\mathcal{Z}^\alpha\nabla(\nabla n\cdot\nabla c)\mathcal{Z}^\alpha\nabla n\,dx
-\int_\Omega\mathcal{Z}^\alpha\nabla(n\Delta c)\mathcal{Z}^\alpha\nabla n\,dx\nonumber\\
&-\int_\Omega\mathcal{Z}^\alpha\nabla(u\cdot\nabla n)\mathcal{Z}^\alpha\nabla n\,dx+\int_{\partial\Omega}\mathcal{Z}^\alpha \Delta n\mathcal{Z}^\alpha \nabla n\cdot\nu\,d\sigma\nonumber\\
&-\int_{\Omega}[\mathcal{Z}^\alpha,\dv]\nabla n\dv(\mathcal{Z}^\alpha \nabla n)\,dx
+\int_{\Omega}[\mathcal{Z}^\alpha,\nabla]\Delta n\cdot\mathcal{Z}^\alpha \nabla n\,dx.
\end{align}

First, applying Young's inequality, we can easily arrive at
\begin{align}\label{3.58}
&\Big{|}\int^t_0\int_{\Omega}[\mathcal{Z}^\alpha,\dv]\nabla n\dv(\mathcal{Z}^\alpha \nabla n)\,dxd\tau-\int^t_0\int_{\Omega}[\mathcal{Z}^\alpha,\nabla]\Delta n\cdot\mathcal{Z}^\alpha \nabla n\,dxd\tau\Big{|}\nonumber\\
\leq&\,\delta\int^t_0\|\nabla\Delta n\|^2_{\mathcal{H}^{m-2}}d\tau+\delta\int^t_0\|\dv(\mathcal{Z}^\alpha \nabla n)\|^2d\tau\nonumber\\
&+D_\delta D_{m+1}\big{(}\int^t_0\|\nabla^2 n\|^2_{\mathcal{H}^{m-2}}d\tau+\int^t_0\|\nabla n\|^2_{\mathcal{H}^{m-1}}d\tau\big{)}.
\end{align}

Next, using Lemma \ref{L2.3}, we have
\begin{align}\label{3.59}
&\Big{|}\int^t_0\int_{\Omega}\mathcal{Z}^\alpha\nabla(\nabla n\cdot\nabla c)\mathcal{Z}^\alpha\nabla n\,dxd\tau\Big{|}\nonumber\\
\leq\,& D_{m+2} P(M(t))\int^t_0\big{(}\|\nabla n\|^2_{\mathcal{H}^{m-1}}
+\|\nabla c\|^2_{\mathcal{H}^{m-1}}+\|\nabla^2 n\|^2_{\mathcal{H}^{m-1}}+\|\nabla^2 c\|^2_{\mathcal{H}^{m-1}}\big{)}\,d\tau.
\end{align}

Furthermore, based on Lemma \ref{L2.3} and Young's inequality, we obtain
\begin{align}\label{3.60}
&\big{|}\int^t_0\int_{\Omega}\mathcal{Z}^\alpha\nabla( n\Delta c)\mathcal{Z}^\alpha\nabla n\,dxd\tau\big{|}\nonumber\\
\leq&\,\delta\int^t_0\|\nabla\Delta c\|^2_{\mathcal{H}^{m-1}}\,d\tau +D_\delta D_{m+2} P(M(t))\nonumber\\
&\times\int^t_0(\|\nabla n\|^2_{\mathcal{H}^{m-1}}+\|\Delta c\|^2_{\mathcal{H}^{m-1}}+\| n\|^2_{\mathcal{H}^{m-1}})\,d\tau.
\end{align}

Now, we turn to estimate the boundary term in \eqref{3.57}. Note that when $|\alpha_0|=m-1$ or $|\alpha_{13}|\neq0$, this term vanishes. So we can integrate by parts along the boundary to deduce that
\begin{align}\label{3.61}
\big{|}\int_0^t\int_{\partial\Omega}\mathcal{Z}^\alpha \Delta n\mathcal{Z}^\alpha \nabla n\cdot\nu\,d\sigma d\tau\big{|}\leq\int_0^t|\partial_t^{\alpha_0}Z_y^{\beta}\Delta n|_{L^2(\partial\Omega)}|\mathcal{Z}^\alpha\nabla n\cdot\nu|_{H^1(\partial\Omega)}d\tau,
\end{align}
where $|\beta|=m-|\alpha_0|-1$. Due to Lemma \ref{L2.2}, we arrive at
\begin{align}\label{3.62}
|\partial_t^{\alpha_0}Z_y^{\beta}\Delta n|_{L^2(\partial\Omega)}\leq D_{m+2}\big{(}\|\nabla\Delta n\|^\frac{1}{2}_{\mathcal{H}^{m-2}}+\|\Delta n\|^\frac{1}{2}_{\mathcal{H}^{m-2}}\big{)}\|\Delta n\|^\frac{1}{2}_{\mathcal{H}^{m-2}}.
\end{align}
With the help of the boundary condition \eqref{1.2.1} and Lemma \ref{L2.2}, we have
\begin{align}\label{3.63}
|\mathcal{Z}^\alpha\nabla n\cdot\nu|_{H^1(\partial\Omega)}\leq D_{m+2}(\|\nabla^2 n\|^\frac{1}{2}_{\mathcal{H}^{m-1}}+\|\nabla n\|^\frac{1}{2}_{\mathcal{H}^{m-1}})\|\nabla n\|^\frac{1}{2}_{\mathcal{H}^{m-1}}.
\end{align}
Based on \eqref{3.61}-\eqref{3.63} and Young's inequality, we can get
\begin{align}\label{3.64}
\big{|}\int_0^t\int_{\partial\Omega}\mathcal{Z}^\alpha \Delta n\mathcal{Z}^\alpha \nabla n\cdot\nu\,d\sigma d\tau\big{|}\leq &\,\delta\int_0^t\|\nabla\Delta n\|^2_{\mathcal{H}^{m-2}}\,d\tau+D_\delta D_{m+2}\int_0^t(\|\nabla n\|^2_{\mathcal{H}^{m-1}}\nonumber\\
&+\|\nabla^2 n\|^2_{\mathcal{H}^{m-1}})\,d\tau.
\end{align}

Finally, we deal with the term involving $u$ in \eqref{3.57}. Integrating by parts leads to that
\begin{align}\label{3.65}
&\int_\Omega\mathcal{Z}^\alpha\nabla(u\cdot\nabla n)\mathcal{Z}^\alpha\nabla n\,dx\nonumber\\
=&\int_{\partial\Omega}\mathcal{Z}^\alpha(u\cdot\nabla n)\mathcal{Z}^\alpha\nabla n\cdot\nu\,dx+\int_{\Omega}\mathcal{Z}^\alpha(u\cdot\nabla n)\dv(\mathcal{Z}^\alpha\nabla n)\,dx\nonumber\\
&+\int_{\Omega}[\mathcal{Z}^\alpha,\nabla](u\cdot\nabla n)\mathcal{Z}^\alpha\nabla n\,dx.
\end{align}
By using Lemmas \ref{L2.3} and \ref{L2.2}, we get
\begin{align}\label{3.66}
&\Big{|}\int^t_0\int_{\partial\Omega}\mathcal{Z}^\alpha(u\cdot\nabla n)\mathcal{Z}^\alpha\nabla n\cdot\nu\,dxd\tau\Big{|}\nonumber\\
\leq&\,D_{m+2}(1+P(M(t)))\int^t_0\big{(}\|\nabla^2 n\|^2_{\mathcal{H}^{m-1}}
+\|\nabla n\|^2_{\mathcal{H}^{m-1}}+\| u\|^2_{\mathcal{H}^{m-1}}+\|\nabla u\|^2_{\mathcal{H}^{m-1}}\big{)}\,d\tau.
\end{align}
Applying Young's inequality and Lemma \ref{L2.3}, we obtain that
\begin{align}\label{3.67}
&\Big{|}\int^t_0\int_{\Omega}\mathcal{Z}^\alpha(u\cdot\nabla n)\dv(\mathcal{Z}^\alpha\nabla n)\,dxd\tau+\int^t_0\int_{\Omega}[\mathcal{Z}^\alpha,\nabla](u\cdot\nabla n)\mathcal{Z}^\alpha\nabla n\,dxd\tau\Big{|}\nonumber\\
\leq&\,\delta\int^t_0\|\dv(\mathcal{Z}^\alpha\nabla n)\|^2\,d\tau
+D_\delta D_{m+2}P(M(t))\int^t_0\big{(}\|\nabla n\|^2_{\mathcal{H}^{m-1}}+\|\nabla^2 n\|^2_{\mathcal{H}^{m-2}}\nonumber\\
&+\| u\|^2_{\mathcal{H}^{m-2}}+\|\nabla u\|^2_{\mathcal{H}^{m-2}}\big{)}d\tau.
\end{align}
Combination of \eqref{3.65}-\eqref{3.67} yields that
\begin{align}\label{3.68}
&\Big{|}\int^t_0\int_\Omega\mathcal{Z}^\alpha\nabla(u\cdot\nabla n)\mathcal{Z}^\alpha\nabla n\,dxd\tau\Big{|}\nonumber\\
\leq&\,\delta\int^t_0\|\dv(\mathcal{Z}^\alpha\nabla n)\|^2\,d\tau+D_{m+2}(1+P(M(t)))\nonumber\\
&\times\int^t_0(\|\nabla^2 n\|^2_{\mathcal{H}^{m-1}}+\|\nabla n\|^2_{\mathcal{H}^{m-1}}
+\| u\|^2_{\mathcal{H}^{m-1}}+\|\nabla u\|^2_{\mathcal{H}^{m-1}})\,d\tau.
\end{align}

Based on the elliptic regularity results with Neumann boundary condition, we obtain that
\begin{align}\label{3.69}
\|\nabla^2 n\|^2_{\mathcal{H}^{m-1}}\leq D_{m+2}(\|\nabla n\|^2_{\mathcal{H}^{m-1}}+\|\Delta n\|^2_{\mathcal{H}^{m-1}}).
\end{align}

Consequently, the combination of \eqref{3.57}-\eqref{3.59}, \eqref{3.64}, \eqref{3.68}-\eqref{3.69} and the inductive assumption yield \eqref{3.52}. Therefore, we complete the proof of Lemma \ref{L3.6}.
\end{proof}

Next, we give the uniform estimate to $\nabla c$.
\begin{Lemma}\label{L3.7} For every $m\geq 0$, a smooth solution of the problem \eqref{1.1.1}-\eqref{1.2} satisfies the estimate
\begin{align}\label{3.70}
&\sup_{0\leq\tau\leq t}\|\nabla c\|^2_{\mathcal{H}^{m}}+\int_0^t\|\Delta c\|^2_{\mathcal{H}^{m}}d\tau\nonumber\\
\leq&\,D_{m+2}\,\Big{\{}\|\nabla c_0\|^2_{\mathcal{H}^{m}}+\delta\int^t_0\|\nabla\Delta c\|^2_{\mathcal{H}^{m-1}}d\tau+\delta\int^t_0\|\nabla n\|^2_{\mathcal{H}^{m}}d\tau\nonumber\\
&\quad\quad\quad+D_\delta\big{(}1+P(M(t))\big{)}\int_0^t\big{(}\|(n,c,u)\|_{\mathcal{H}^{m}}^2+\|\nabla c\|_{\mathcal{H}^{m}}^2+\|\nabla u\|_{\mathcal{H}^{m-1}}^2\nonumber\\
&\quad\quad\quad+\|\nabla^2c\|_{\mathcal{H}^{m-1}}^2\big{)}d\tau\Big{\}},
\end{align}
where $\delta>0$ is a small enough constant.
\end{Lemma}

\begin{proof}
Multiplying \eqref{1.1.2} by $\Delta n$ yields that
\begin{align}
\int_\Omega c_t\Delta c\,dx+\int_\Omega u\cdot\nabla c\Delta c\,dx=\int_\Omega\Delta c\Delta c\,dx-\int_\Omega nc\Delta c\,dx.
\end{align}
Integration by parts and the boundary condition \eqref{1.2.1} lead to
\begin{align}
\int_\Omega c_t\Delta c\,dx=-\frac{1}{2}\frac{d}{dt}\int_\Omega |\nabla c|^2\,dx.
\end{align}
By using Young's inequality, we arrive at
\begin{align}
\sup_{0\leq\tau\leq t}\|\nabla c\|^2+\int_0^t\|\Delta c\|^2\,d\tau\leq D\Big{\{}\|\nabla c_0\|^2+M(t)\int_0^t(\|\nabla c\|^2+\| c\|^2)\,d\tau\Big{\}}.
\end{align}
Hence, \eqref{3.70} holds in the case of $|\alpha|=0$.

Next, we deal with the higher order estimates. Assume that \eqref{3.70} has been proved for $|\alpha|\leq m-1$, we need to show it still holds for $|\alpha|=m$. By applying $\mathcal{Z}^\alpha\nabla$ with $|\alpha|=m$ to \eqref{1.1.2}, we obtain that
\begin{align}\label{3.74}
 \mathcal{Z}^\alpha\nabla c_t-\mathcal{Z}^\alpha\nabla\Delta n=-\mathcal{Z}^\alpha\nabla( nc)-\mathcal{Z}^\alpha\nabla(u\cdot\nabla c).
\end{align}
By multiplying \eqref{3.74} by $\mathcal{Z}^\alpha\nabla c$, we obtain that
\begin{align*}
\frac{1}{2}\frac{d}{dt}\int_\Omega|\mathcal{Z}^\alpha\nabla c|^2\,dx-\int_\Omega\mathcal{Z}^\alpha\nabla \Delta c\mathcal{Z}^\alpha\nabla  c\,dx=&-\int_\Omega\mathcal{Z}^\alpha\nabla(nc)\mathcal{Z}^\alpha\nabla c\,dx\nonumber\\
&-\int_\Omega\mathcal{Z}^\alpha\nabla(u\cdot\nabla c)\mathcal{Z}^\alpha\nabla c\,dx.
\end{align*}
By integrating by parts and using the boundary condition \eqref{1.2.1}, we obtain that
\begin{align*}
&-\int_\Omega\mathcal{Z}^\alpha\nabla \Delta c\mathcal{Z}^\alpha\nabla  c\,dx\nonumber\\
=&-\int_{\partial\Omega}\mathcal{Z}^\alpha \Delta c\mathcal{Z}^\alpha \nabla c\cdot\nu\,d\sigma+\int_{\Omega}|\dv(\mathcal{Z}^\alpha \nabla c)|^2\,dx\nonumber\\
&+\int_{\Omega}[\mathcal{Z}^\alpha,\dv]\nabla c\dv(\mathcal{Z}^\alpha \nabla c)\,dx-\int_{\Omega}[\mathcal{Z}^\alpha,\nabla]\Delta c\cdot\mathcal{Z}^\alpha \nabla c\,dx.
\end{align*}
Hence, we have
\begin{align}\label{3.75}
&\frac{1}{2}\frac{d}{dt}\int_\Omega|\mathcal{Z}^\alpha\nabla c|^2\,dx+\int_{\Omega}|\dv(\mathcal{Z}^\alpha \nabla c)|^2\,dx\nonumber\\
=&-\int_\Omega\mathcal{Z}^\alpha\nabla(nc)\mathcal{Z}^\alpha\nabla c\,dx
-\int_\Omega\mathcal{Z}^\alpha\nabla(u\cdot\nabla c)\mathcal{Z}^\alpha\nabla c\,dx+\int_{\partial\Omega}\mathcal{Z}^\alpha \Delta c\mathcal{Z}^\alpha \nabla c\cdot\nu\,d\sigma\nonumber\\
&-\int_{\Omega}[\mathcal{Z}^\alpha,\dv]\nabla c\dv(\mathcal{Z}^\alpha \nabla c)\,dx
+\int_{\Omega}[\mathcal{Z}^\alpha,\nabla]\Delta c\cdot\mathcal{Z}^\alpha \nabla c\,dx.
\end{align}

First, similar to Lemma \ref{L3.6}, we can easily arrive at
\begin{align}\label{3.76}
&\Big{|}\int^t_0\int_{\Omega}[\mathcal{Z}^\alpha,\dv]\nabla c\dv(\mathcal{Z}^\alpha \nabla c)\,dxd\tau-\int^t_0\int_{\Omega}[\mathcal{Z}^\alpha,\nabla]\Delta c\cdot\mathcal{Z}^\alpha \nabla c\,dxd\tau\Big{|}\nonumber\\
\leq&\,\delta\int^t_0\|\nabla\Delta c\|^2_{\mathcal{H}^{m-1}}d\tau+\delta\int^t_0\|\dv(\mathcal{Z}^\alpha \nabla c)\|^2d\tau\nonumber\\
&+D_\delta D_{m+2}\Big{(}\int^t_0\|\nabla^2 c\|^2_{\mathcal{H}^{m-1}}d\tau+\int^t_0\|\nabla c\|^2_{\mathcal{H}^{m}}d\tau\Big{)}.
\end{align}

Next, using Lemma \ref{L2.3} and Young's inequality, we have
\begin{align}\label{3.77}
&\Big{|}\int^t_0\int_{\Omega}\mathcal{Z}^\alpha\nabla(nc)\mathcal{Z}^\alpha\nabla c\,dxd\tau\Big{|}\nonumber\\
\leq&\,\delta\int^t_0\|\nabla n\|^2_{\mathcal{H}^{m}} d\tau +D_\delta D_{m+2} P(M(t))\int^t_0(\|c\|^2_{\mathcal{H}^{m}}+\|n\|^2_{\mathcal{H}^{m}}+\|\nabla c\|^2_{\mathcal{H}^{m}})\,d\tau.
\end{align}

Now, we turn to estimate the boundary term in \eqref{3.75}. Like in Lemma \ref{L3.6}, when $|\alpha_0|=m-1$ or $|\alpha_{13}|\neq0$, this term vanishes. Thus, integrating by parts along the boundary lead to
\begin{align}\label{3.78}
\big{|}\int_0^t\int_{\partial\Omega}\mathcal{Z}^\alpha \Delta c\mathcal{Z}^\alpha \nabla c\cdot\nu\,d\sigma d\tau\big{|}\leq\int_0^t|\partial_t^{\alpha_0}Z_y^{\beta}\Delta c|_{L^2(\partial\Omega)}|\mathcal{Z}^\alpha\nabla c\cdot\nu|_{H^1(\partial\Omega)}d\tau,
\end{align}
where $|\beta|=m-|\alpha_0|-1$. By virtue of Lemma \ref{L2.2}, we arrive at
\begin{align}\label{3.79}
|\partial_t^{\alpha_0}Z_y^{\beta}\Delta c|_{L^2(\partial\Omega)}\leq D(\|\nabla\Delta c\|^\frac{1}{2}_{\mathcal{H}^{m-1}}+\|\Delta c\|^\frac{1}{2}_{\mathcal{H}^{m-1}})\|\Delta c\|^\frac{1}{2}_{\mathcal{H}^{m-1}}.
\end{align}
Based on the boundary condition \eqref{1.2.1} and Lemma \ref{L2.2}, we can deduce that
\begin{align}\label{3.80}
|\mathcal{Z}^\alpha\nabla c\cdot\nu|_{H^1(\partial\Omega)}\leq D(\|\nabla^2 c\|^\frac{1}{2}_{\mathcal{H}^{m}}+\|\nabla c\|^\frac{1}{2}_{\mathcal{H}^{m}})\|\nabla c\|^\frac{1}{2}_{\mathcal{H}^{m}}.
\end{align}
Combining \eqref{3.78}-\eqref{3.80} and using Young's inequality, we can get
\begin{align}\label{3.81}
\Big{|}\int_0^t\int_{\partial\Omega}\mathcal{Z}^\alpha \Delta c\mathcal{Z}^\alpha \nabla c\cdot\nu\,d\sigma d\tau\Big{|}\leq &\,\delta\int_0^t\|\nabla\Delta c\|^2_{\mathcal{H}^{m-1}}\,d\tau+\delta\int_0^t\|\nabla^2 c\|^2_{\mathcal{H}^{m}}\,d\tau\nonumber\\
&+D_\delta D_{m+2}\int_0^t(\|\nabla c\|^2_{\mathcal{H}^{m}}+\|\nabla^2 c\|^2_{\mathcal{H}^{m-1}})\,d\tau.
\end{align}

Finally, we deal with the term involving $u$ in \eqref{3.75}. By using integration by parts, it is easy to deduce that
\begin{align}\label{3.82}
&\int_\Omega\mathcal{Z}^\alpha\nabla(u\cdot\nabla c)\mathcal{Z}^\alpha\nabla c\,dx\nonumber\\
=&\int_{\partial\Omega}\mathcal{Z}^\alpha(u\cdot\nabla c)\mathcal{Z}^\alpha\nabla c\cdot\nu\,dx+\int_{\Omega}\mathcal{Z}^\alpha(u\cdot\nabla c)\dv(\mathcal{Z}^\alpha\nabla c)\,dx\nonumber\\
&+\int_{\Omega}[\mathcal{Z}^\alpha,\nabla](u\cdot\nabla c)\mathcal{Z}^\alpha\nabla c\,dx.
\end{align}
Similar to \eqref{3.61}, we can integrate by parts along the boundary to deduce that
\begin{align}\label{3.83}
\Big{|}\int_{\partial\Omega}\mathcal{Z}^\alpha(u\cdot\nabla c)\mathcal{Z}^\alpha\nabla c\cdot\nu\,dx\Big{|}\leq D|\partial_t^{\alpha_0}Z_y^\beta(u\cdot\nabla c)|_{L^2(\partial\Omega)}|\mathcal{Z}^\alpha\nabla c\cdot\nu|_{H^1{\partial\Omega}},
\end{align}
where $|\beta|=m-|\alpha_0|-1$. Applying Lemma \ref{L2.2}, we arrive at
\begin{align*}
|\partial_t^{\alpha_0}Z_y^\beta(u\cdot\nabla c)|_{L^2(\partial\Omega)}&\leq D(\|\nabla(u\cdot\nabla c)\|_{\mathcal{H}^{m-1}}^{\frac{1}{2}}+\|u\cdot\nabla c\|_{\mathcal{H}^{m-1}}^{\frac{1}{2}})\|u\cdot\nabla c\|_{\mathcal{H}^{m-1}}^{\frac{1}{2}},\\
|\mathcal{Z}^\alpha\nabla c\cdot\nu|_{H^1{\partial\Omega}}&\leq D(\|\nabla^2 c\|_{\mathcal{H}^{m}}^{\frac{1}{2}}+\|\nabla c\|_{\mathcal{H}^{m}}^{\frac{1}{2}})\|\nabla c\|_{\mathcal{H}^{m}}^{\frac{1}{2}}.
\end{align*}
By using Lemma \ref{L2.3} and Young's inequality, we obtain that
\begin{align}\label{3.84}
&\Big{|}\int^t_0\int_{\partial\Omega}\mathcal{Z}^\alpha(u\cdot\nabla c)\mathcal{Z}^\alpha\nabla c\cdot\nu\,dxd\tau\Big{|}\nonumber\\
\leq&\, \delta\int^t_0 \|\nabla^2 c\|_{\mathcal{H}^{m}}^2d\tau+(1+P(M(t)))\nonumber\\
&\times\int^t_0 (\|\nabla c\|_{\mathcal{H}^{m}}^2+\|\nabla^2 c\|_{\mathcal{H}^{m-1}}^2+\|u\|_{\mathcal{H}^{m-1}}^2+\|\nabla u\|_{\mathcal{H}^{m-1}}^2)\,d\tau.
\end{align}
In view of Young's inequality, we can get
\begin{align}\label{3.85}
&\Big{|}\int^t_0\int_\Omega\mathcal{Z}^\alpha(u\cdot\nabla c)\dv(\mathcal{Z}^\alpha\nabla c)\,dxd\tau+\int^t_0\int_{\Omega}[\mathcal{Z}^\alpha,\nabla](u\cdot\nabla c)\mathcal{Z}^\alpha\nabla c\,dxd\tau\Big{|}\nonumber\\
\leq& \,\delta\int^t_0\|\dv(\mathcal{Z}^\alpha\nabla c)\|^2\,d\tau
+D(1+P(M(t)))\int^t_0 \big{(}\|\nabla c\|_{\mathcal{H}^{m}}^2+\|\nabla^2 c\|_{\mathcal{H}^{m-1}}^2\nonumber\\
&+\|u\|_{\mathcal{H}^{m}}^2+\|\nabla u\|_{\mathcal{H}^{m-1}}^2\big{)}\,d\tau.
\end{align}
The combination of \eqref{3.84} and \eqref{3.85} yields that
\begin{align}\label{3.86}
&\Big{|}\int^t_0\int_\Omega\mathcal{Z}^\alpha\nabla(u\cdot\nabla c)\mathcal{Z}^\alpha\nabla c\,dxd\tau\Big{|}\nonumber\\
\leq&\,\delta\int^t_0\|\dv(\mathcal{Z}^\alpha\nabla c)\|^2\,d\tau+\delta\int^t_0 \|\nabla^2 c\|_{\mathcal{H}^{m}}^2d\tau+D_{m+2}\big{(}1+P(M(t))\big{)}\nonumber\\
&\times\int^t_0\big{(}\|\nabla c\|^2_{\mathcal{H}^{m}}+\|\nabla^2 c\|^2_{\mathcal{H}^{m-1}}
+\| u\|^2_{\mathcal{H}^{m}}+\|\nabla u\|^2_{\mathcal{H}^{m-1}}\big{)}\,d\tau
\end{align}

Consequently, based on \eqref{3.69}, \eqref{3.76}, \eqref{3.77}, \eqref{3.81}, \eqref{3.86} and the inductive assumption yield \eqref{3.70}. Therefore, we complete the proof of Lemma \ref{L3.7}.
\end{proof}

\subsection{Conormal Energy Estimates for $\Delta n$ and $\Delta c$}
In this subsection, we can establish some uniform estimates for $\|\Delta n\|_{\mathcal{H}^{m-1}}$ and $\|\Delta c\|_{\mathcal{H}^{m-1}}$. We have the following uniform estimate with respect to $n$.
\begin{Lemma}\label{L3.8} For every $m\geq 1$, a smooth solution of the problem \eqref{1.1.1}-\eqref{1.2} satisfies the estimate
\begin{align}\label{3.87}
&\sup_{0\leq\tau\leq t}\|\Delta n\|^2_{\mathcal{H}^{m-1}}+\int_0^t\|\nabla\Delta n\|^2_{\mathcal{H}^{m-1}}d\tau\nonumber\\
\leq&\,D_{m+2}\,\Big{\{}\|\Delta n_0\|^2_{\mathcal{H}^{m-1}}+\big{(}1+P(M(t))\big{)}\int_0^t\big{(}\|(n,c,u)\|_{\mathcal{H}^{m-1}}^2\nonumber\\
&\quad\quad\quad+\|\nabla(n,u)\|_{\mathcal{H}^{m-1}}^2+\|\nabla c\|_{\mathcal{H}^{m}}^2
+\|\nabla^2 (n,c) \|_{\mathcal{H}^{m-1}}^2\big{)}d\tau\Big{\}}.
\end{align}
\end{Lemma}

\begin{proof}
Applying $\nabla$ to the equation \eqref{1.1.1} gives
\begin{align}\label{3.88}
\nabla n_t+\nabla(u\cdot\nabla n)=\nabla\Delta n-\nabla(n\Delta c)-\nabla(\nabla n\cdot\nabla c).
\end{align}
By multiplying \eqref{3.88} by $\nabla\Delta n$ and integrating over $\Omega$, we obtain that
\begin{align}\label{3.89}
&-\int_{\Omega}\nabla n_t\cdot\nabla\Delta n\,dx+\int_{\Omega}|\nabla\Delta n|^2\,dx\nonumber\\
=&\int_{\Omega}\nabla(u\cdot\nabla n)\cdot\nabla\Delta n\,dx+\int_{\Omega}\nabla(n\Delta c)\cdot\nabla\Delta n\,dx\nonumber\\
&+\int_{\Omega}\nabla(\nabla n\cdot\nabla c)\cdot\nabla\Delta n\,dx.
\end{align}
By using the integration by parts and the boundary condition \eqref{1.2.1}, we get
\begin{align}\label{3.90}
-\int_{\Omega}\nabla n_t\cdot\nabla\Delta n\,dx=\frac{1}{2}\frac{d}{dt}\int_{\Omega}|\Delta n|^2\,dx.
\end{align}
In view of Young's inequality, we obtain that
\begin{align}\label{3.91}
&\Big{|}\int_0^t\int_{\Omega}\nabla(u\cdot\nabla n)\cdot\nabla\Delta n-\nabla(n\Delta c)\cdot\nabla\Delta n-\nabla(\nabla n\cdot\nabla c)\cdot\nabla\Delta n\,dxd\tau\Big{|}\nonumber\\
\leq&\,\delta\int^t_0\|\nabla\Delta n\|^2\,d\tau+D_\delta M(t)\int^t_0\big{(}\|n\|^2+\|\nabla n\|^2+\|\nabla^2n\|^2\big{)}\,d\tau,
\end{align}
where $\delta$ is a small enough constant.
Therefore, based on \eqref{3.89}-\eqref{3.91}, \eqref{3.87} holds for $m=1$.

Now, we turn to do higher order uniform estimates. Assume that \eqref{3.87} has been proved for $|\alpha|\leq m-2$, we need to show that it holds for $|\alpha|=m-1$. By applying $\mathcal{Z}^\alpha$ with $|\alpha|=m-1$ to \eqref{3.88}, we obtain that
\begin{align}\label{3.92}
\mathcal{Z}^\alpha\nabla n_t+\mathcal{Z}^\alpha\nabla(u\cdot\nabla n)=\mathcal{Z}^\alpha\nabla\Delta n-\mathcal{Z}^\alpha\nabla(n\Delta c)-\mathcal{Z}^\alpha\nabla(\nabla n\cdot\nabla c).
\end{align}
Multiplying \eqref{3.92} by $\nabla\mathcal{Z}^\alpha\Delta n$, we get
\begin{align}\label{3.93}
&-\int_{\Omega}\mathcal{Z}^\alpha\nabla n_t\cdot\nabla\mathcal{Z}^\alpha\Delta n\,dx+\int_{\Omega}\mathcal{Z}^\alpha\nabla\Delta n\cdot\nabla\mathcal{Z}^\alpha\Delta n\,dx\nonumber\\
=&\int_{\Omega}\mathcal{Z}^\alpha\nabla(u\cdot\nabla n)\cdot\nabla\mathcal{Z}^\alpha\Delta n\,dx
+\int_{\Omega}\mathcal{Z}^\alpha\nabla(n\Delta c)\cdot\nabla\mathcal{Z}^\alpha\Delta n\,dx\nonumber\\
&+\int_{\Omega}\mathcal{Z}^\alpha\nabla(\nabla n\cdot\nabla c)\cdot\nabla\mathcal{Z}^\alpha\Delta n\,dx.
\end{align}
Due to integration by parts, we have
\begin{align}
-\int_{\Omega}\mathcal{Z}^\alpha\nabla n_t\cdot\nabla\mathcal{Z}^\alpha\Delta n\,dx=&\frac{1}{2}\frac{d}{dt}\int_{\Omega}|\mathcal{Z}^\alpha\Delta n|^2\,dx-\int_{\partial\Omega}\nu\cdot\mathcal{Z}^\alpha\nabla n^\epsilon_t\mathcal{Z}^\alpha\Delta n\,d\sigma\nonumber\\
&-\int_{\Omega}[\mathcal{Z}^\alpha,\nabla\cdot]\nabla n_t\mathcal{Z}^\alpha\Delta n\,dx,\label{3.94}\\
\int_{\Omega}\mathcal{Z}^\alpha\nabla\Delta n\cdot\nabla\mathcal{Z}^\alpha\Delta n\,dx=&\int_{\Omega}|\nabla\mathcal{Z}^\alpha\Delta n|^2\,dx+\int_{\Omega}[\mathcal{Z}^\alpha,\nabla]\Delta n\cdot\nabla\mathcal{Z}^\alpha\Delta n\,dx.\label{3.95}
\end{align}
Hence, from \eqref{3.94} and \eqref{3.95}, we get
\begin{align}\label{3.96}
&\frac{1}{2}\frac{d}{dt}\int_{\Omega}|\mathcal{Z}^\alpha\Delta n|^2\,dx+\int_{\Omega}|\nabla\mathcal{Z}^\alpha\Delta n|^2\,dx\nonumber\\
=&\int_{\Omega}\mathcal{Z}^\alpha\nabla(u\cdot\nabla n)\cdot\nabla\mathcal{Z}^\alpha\Delta n\,dx+\int_{\Omega}\mathcal{Z}^\alpha\nabla(n\Delta c)\cdot\nabla\mathcal{Z}^\alpha\Delta n\,dx\nonumber\\
&+\int_{\partial\Omega}\nu\cdot\mathcal{Z}^\alpha\nabla n_t\mathcal{Z}^\alpha\Delta n\,d\sigma+\int_{\Omega}[\mathcal{Z}^\alpha,\nabla\cdot]\nabla n_t\mathcal{Z}^\alpha\Delta n\,dx\nonumber\\
&-\int_{\Omega}[\mathcal{Z}^\alpha,\nabla]\Delta n\cdot\nabla\mathcal{Z}^\alpha\Delta n\,dx+\int_{\Omega}\mathcal{Z}^\alpha\nabla(\nabla n\cdot\nabla c)\cdot\nabla\mathcal{Z}^\alpha\Delta n\,dx.
\end{align}

First, we deal with the boundary term on the right-hand side of \eqref{3.96}. Note that when $|\alpha_0|=m-1$, this integral vanishes. Hence, we assume $|\alpha_0|\leq m-2$. It is easy to deduce that
\begin{align}
\big{|}\int^t_0\int_{\partial\Omega}\nu\cdot\mathcal{Z}^\alpha\nabla n_t\mathcal{Z}^\alpha\Delta n\,d\sigma d\tau\big{|}\leq\int^t_0|\nu\cdot\mathcal{Z}^\alpha\nabla n_t|_{L^2(\partial\Omega)}|\mathcal{Z}^\alpha\Delta n|_{L^2(\partial\Omega)}\,d\tau.
\end{align}
Based on Lemma \ref{L2.2} and the boundary condition \eqref{1.2.1}, we get
\begin{align}
|\nu\cdot\mathcal{Z}^\alpha\nabla n_t|_{L^2(\partial\Omega)}\leq&\, D_m(\|\nabla^2 n\|^\frac{1}{2}_{\mathcal{H}^{m-1}}+\|\nabla n\|_{\mathcal{H}^{m-1}}^\frac{1}{2})\|\nabla n\|^\frac{1}{2}_{\mathcal{H}^{m-1}}\nonumber,\\
|\mathcal{Z}^\alpha\Delta n|_{L^2(\partial\Omega)}\leq&\, D_{m}(\|\nabla \Delta n\|^\frac{1}{2}_{\mathcal{H}^{m-1}}+\|\Delta n\|_{\mathcal{H}^{m-1}}^\frac{1}{2})\|\Delta n\|^\frac{1}{2}_{\mathcal{H}^{m-1}}\nonumber.
\end{align}
Therefore, in view of Young's inequality, we have
\begin{align}\label{3.98}
\big{|}\int^t_0\int_{\partial\Omega}\nu\cdot\mathcal{Z}^\alpha\nabla n_t\mathcal{Z}^\alpha\Delta n\,d\sigma d\tau\big{|}\leq&\,\delta\int^t_0\|\nabla \Delta n\|^2_{\mathcal{H}^{m-1}}\,d\tau+D_\delta D_{m}\int^t_0\big{(}\|\nabla^2 n\|^2_{\mathcal{H}^{m-1}}\nonumber\\
&+\|\nabla n\|^2_{\mathcal{H}^{m-1}}\big{)}\,d\tau.
\end{align}

Next, we can use Young's inequality to get the following estimate directly,
\begin{align}\label{3.99}
&\Big{|}\int^t_0\int_{\Omega}[\mathcal{Z}^\alpha,\nabla]\Delta n\cdot\nabla\mathcal{Z}^\alpha\Delta n\,dxd\tau\Big{|}\nonumber\\
\leq&\,\delta\int^t_0\|\nabla \Delta n\|^2_{\mathcal{H}^{m-1}}\,d\tau
+D_\delta\int^t_0\|\nabla\Delta n\|^2_{\mathcal{H}^{m-2}}\,d\tau.
\end{align}
Also, it is easy to deduce that
\begin{align}\label{3.100}
\Big{|}\int^t_0\int_{\Omega}[\mathcal{Z}^\alpha,\nabla\cdot]\nabla n_t\mathcal{Z}^\alpha\Delta n\,dxd\tau\Big{|}\leq D\int^t_0\|\nabla^2 n\|^2_{\mathcal{H}^{m-1}}\,d\tau.
\end{align}

Furthermore, by virtue of Lemma \ref{L2.3} and Young's inequality, we obtain that
\begin{align}
&\Big{|}\int^t_0\int_{\Omega}\mathcal{Z}^\alpha\nabla(\nabla n\cdot\nabla c)\cdot\nabla\mathcal{Z}^\alpha\Delta n\,dxd\tau\Big{|}\nonumber\\
\leq&\,\delta\int^t_0\|\nabla \Delta n\|^2_{\mathcal{H}^{m-1}}\,d\tau+D_\delta P(M(t))\int^t_0\big{(}\|\nabla^2 n\|^2_{\mathcal{H}^{m-1}}+\|\nabla n\|^2_{\mathcal{H}^{m-1}}\nonumber\\
&+\|\nabla^2 c\|^2_{\mathcal{H}^{m-1}}+\|\nabla c\|^2_{\mathcal{H}^{m-1}}\big{)}\,d\tau,\label{3.101}\\
&\Big{|}\int^t_0\int_{\Omega}\mathcal{Z}^\alpha\nabla(u\cdot\nabla n)\cdot\nabla\mathcal{Z}^\alpha\Delta n\,dxd\tau\Big{|}\nonumber\\
\leq&\,\delta\int^t_0\|\nabla \Delta n\|^2_{\mathcal{H}^{m-1}}\,d\tau+D_\delta P(M(t))\int^t_0\big{(}\|\nabla n\|^2_{\mathcal{H}^{m-1}}+\|\nabla^2 n\|^2_{\mathcal{H}^{m-1}}\nonumber\\
&+\|\nabla u\|^2_{\mathcal{H}^{m-1}}+\|u\|^2_{\mathcal{H}^{m-1}}\big{)}\,d\tau.\label{3.102}
\end{align}

Finally, because we don't expect that $\|\nabla\Delta c\|^2_{\mathcal{H}^{m-1}}$ appear in the right-hand side of energy inequality, we deal with $\int_{\Omega}\mathcal{Z}^\alpha\nabla(n\Delta c)\cdot\nabla\mathcal{Z}^\alpha\Delta n\,dx$ with the help of the equation \eqref{1.1.2} and Lemma \ref{L2.3}. We have
\begin{align}
&\Big{|}\int^t_0\int_{\Omega}\mathcal{Z}^\alpha\nabla(n\Delta c)\cdot\nabla\mathcal{Z}^\alpha\Delta n\,dxd\tau\Big{|}\nonumber\\
\leq&\,\delta\int^t_0\|\nabla \Delta n\|^2_{\mathcal{H}^{m-1}}\,d\tau+D_\delta P(M(t))\int^t_0\big{(}\| (n,c,u)\|^2_{\mathcal{H}^{m-1}}\nonumber\\
&+\|\nabla( n,u)\|^2_{\mathcal{H}^{m-1}}+\|\nabla^2 c\|^2_{\mathcal{H}^{m-1}}+\|\nabla c\|^2_{\mathcal{H}^{m}}\big{)}\,d\tau.\label{3.103}
\end{align}
Consequently, the combination of \eqref{3.69}, \eqref{3.98}-\eqref{3.103} and the inductive assumption yield \eqref{3.87}. Therefore, we complete the proof of Lemma \ref{L3.8}.
\end{proof}

Now, we turn to the estimate $\|\Delta c\|_{\mathcal{H}^{m-1}}$. We have
\begin{Lemma}\label{L3.9} For every $m\geq 1$, a smooth solution of the problem \eqref{1.1.1}-\eqref{1.2} satisfies the estimate
\begin{align}\label{3.104}
&\sup_{0\leq\tau\leq t}\|\Delta c\|^2_{\mathcal{H}^{m-1}}+\int_0^t\|\nabla\Delta c\|^2_{\mathcal{H}^{m-1}}d\tau\nonumber\\
\leq&\,D_{m+2}\,\Big{\{}\|\Delta c_0\|^2_{\mathcal{H}^{m-1}}+\big{(}1+P(M(t))\big{)}\int_0^t\big{(}\|(n,c,u)\|_{\mathcal{H}^{m-1}}^2\nonumber\\
&+\|\nabla(n,c,u)\|_{\mathcal{H}^{m-1}}^2
+\|\nabla^2 c \|_{\mathcal{H}^{m-1}}^2\big{)}d\tau\Big{\}}.
\end{align}
\end{Lemma}

\begin{proof}
Applying $\nabla$ to the equation \eqref{1.1.2} gives
\begin{align}\label{3.105}
\nabla c_t+\nabla(u\cdot\nabla c)=\nabla\Delta c-\nabla(nc).
\end{align}
By multiplying \eqref{3.105} by $\nabla\Delta c$ and integrating over $\Omega$, we get
\begin{align}\label{3.106}
-\int_{\Omega}\nabla c_t\cdot\nabla\Delta c\,dx+\int_{\Omega}|\nabla\Delta c|^2\,dx=\int_{\Omega}\nabla(u\cdot\nabla c)\cdot\nabla\Delta c\,dx
+\int_{\Omega}\nabla(nc)\cdot\nabla\Delta c\,dx.
\end{align}
Integrating by parts and using the boundary condition \eqref{1.2.1} yield
\begin{align}\label{3.107}
-\int_{\Omega}\nabla c_t\cdot\nabla\Delta c\,dx=\frac{1}{2}\frac{d}{dt}\int_{\Omega}|\Delta c|^2\,dx.
\end{align}
By virtue of Young's inequality, we get
\begin{align}\label{3.108}
&\Big{|}\int_0^t\int_{\Omega}\nabla(u\cdot\nabla c)\cdot\nabla\Delta c-\nabla(nc)\cdot\nabla\Delta c\,dxd\tau\Big{|}\nonumber\\
\leq&\,\delta\int^t_0\|\nabla\Delta c\|^2\,d\tau+D_\delta M(t)\int^t_0(\|c\|^2+\|\nabla c\|^2+\|\nabla^2c\|^2)\,d\tau.
\end{align}
Therefore, the combination of \eqref{3.106}-\eqref{3.108} implies that \eqref{3.104} holds for $m=1$.

Now, we show that \eqref{3.104} holds for $|\alpha|\leq m-1$ . Assume that \eqref{3.104} is proved for $|\alpha|\leq m-2$, we need to prove it for $|\alpha|=m-1$. Applying $\mathcal{Z}^\alpha$ with $|\alpha|=m-1$ to \eqref{3.105} gives
\begin{align}\label{3.109}
\mathcal{Z}^\alpha\nabla c_t+\mathcal{Z}^\alpha\nabla(u\cdot\nabla c)=\mathcal{Z}^\alpha\nabla\Delta c-\mathcal{Z}^\alpha\nabla(nc).
\end{align}
Multiplying \eqref{3.109} by $\nabla\mathcal{Z}^\alpha\Delta c$, we obtain that
\begin{align}\label{3.110}
&-\int_{\Omega}\mathcal{Z}^\alpha\nabla c_t\cdot\nabla\mathcal{Z}^\alpha\Delta c\,dx+\int_{\Omega}\mathcal{Z}^\alpha\nabla\Delta c\cdot\nabla\mathcal{Z}^\alpha\Delta c\,dx\nonumber\\
=&\int_{\Omega}\mathcal{Z}^\alpha\nabla(u\cdot\nabla c)\cdot\nabla\mathcal{Z}^\alpha\Delta c\,dx
+\int_{\Omega}\mathcal{Z}^\alpha\nabla(nc)\cdot\nabla\mathcal{Z}^\alpha\Delta c\,dx.
\end{align}
By integrating by parts, we have
\begin{align}
-\int_{\Omega}\mathcal{Z}^\alpha\nabla c_t\cdot\nabla\mathcal{Z}^\alpha\Delta c\,dx=&\frac{1}{2}\frac{d}{dt}\int_{\Omega}|\mathcal{Z}^\alpha\Delta c|^2\,dx-\int_{\partial\Omega}\nu\cdot\mathcal{Z}^\alpha\nabla c^\epsilon_t\mathcal{Z}^\alpha\Delta c\,d\sigma\nonumber\\
&-\int_{\Omega}[\mathcal{Z}^\alpha,\nabla\cdot]\nabla c_t\mathcal{Z}^\alpha\Delta c\,dx,\label{3.111}\\
\int_{\Omega}\mathcal{Z}^\alpha\nabla\Delta c\cdot\nabla\mathcal{Z}^\alpha\Delta c\,dx=&\int_{\Omega}|\nabla\mathcal{Z}^\alpha\Delta c|^2\,dx+\int_{\Omega}[\mathcal{Z}^\alpha,\nabla]\Delta c\cdot\nabla\mathcal{Z}^\alpha\Delta c\,dx.\label{3.112}
\end{align}
From \eqref{3.111} and \eqref{3.112}, we get
\begin{align}\label{3.113}
&\frac{1}{2}\frac{d}{dt}\int_{\Omega}|\mathcal{Z}^\alpha\Delta c|^2\,dx+\int_{\Omega}|\nabla\mathcal{Z}^\alpha\Delta c|^2\,dx\nonumber\\
=&\int_{\Omega}\mathcal{Z}^\alpha\nabla(u\cdot\nabla c)\cdot\nabla\mathcal{Z}^\alpha\Delta c\,dx+\int_{\Omega}\mathcal{Z}^\alpha\nabla(n c)\cdot\nabla\mathcal{Z}^\alpha\Delta c\,dx\nonumber\\
&+\int_{\partial\Omega}\nu\cdot\mathcal{Z}^\alpha\nabla c_t\mathcal{Z}^\alpha\Delta c\,d\sigma+\int_{\Omega}[\mathcal{Z}^\alpha,\nabla\cdot]\nabla c_t\mathcal{Z}^\alpha\Delta c\,dx\nonumber\\
&-\int_{\Omega}[\mathcal{Z}^\alpha,\nabla]\Delta c\cdot\nabla\mathcal{Z}^\alpha\Delta c\,dx.
\end{align}

First, we estimate the boundary term in the right-hand side of \eqref{3.113}. Note that when $|\alpha_0|=m-1$, this integral vanishes. Hence, we assume $|\alpha_0|\leq m-2$. It is easy to deduce that
\begin{align}
\Big{|}\int^t_0\int_{\partial\Omega}\nu\cdot\mathcal{Z}^\alpha\nabla c_t\mathcal{Z}^\alpha\Delta c\,d\sigma d\tau\Big{|}\leq\int^t_0|\nu\cdot\mathcal{Z}^\alpha\nabla c_t|_{L^2(\partial\Omega)}|\mathcal{Z}^\alpha\Delta c|_{L^2(\partial\Omega)}\,d\tau.\nonumber
\end{align}
Similar to \eqref{3.98}, by virtue of Lemma \ref{L2.2}, the boundary condition \eqref{1.2.1} and Young's inequality,  we get
\begin{align}\label{3.115}
&\Big{|}\int^t_0\int_{\partial\Omega}\nu\cdot\mathcal{Z}^\alpha\nabla c_t\mathcal{Z}^\alpha\Delta c\,d\sigma d\tau\Big{|}\nonumber\\
\leq&\,\delta\int^t_0\|\nabla \Delta c\|^2_{\mathcal{H}^{m-1}}\,d\tau
+D_\delta D_{m}\int^t_0(\|\nabla^2 c\|^2_{\mathcal{H}^{m-1}}+\|\nabla c\|^2_{\mathcal{H}^{m-1}})\,d\tau.
\end{align}

Next, by using Young's inequality, we can easy obtain
\begin{align}\label{3.116}
&\Big{|}\int^t_0\int_{\Omega}[\mathcal{Z}^\alpha,\nabla]\Delta c\cdot\nabla\mathcal{Z}^\alpha\Delta c\,dxd\tau\Big{|}\nonumber\\
\leq\,&\delta\int^t_0\|\nabla \Delta c\|^2_{\mathcal{H}^{m-1}}\,d\tau
+D_\delta\int^t_0\|\nabla\Delta c\|^2_{\mathcal{H}^{m-2}}\,d\tau.
\end{align}
Also, it is easy to deduce that
\begin{align}\label{3.117}
\Big{|}\int^t_0\int_{\Omega}[\mathcal{Z}^\alpha,\nabla\cdot]\nabla c_t\mathcal{Z}^\alpha\Delta c\,dxd\tau\Big{|}\leq D\int^t_0\|\nabla^2 c\|^2_{\mathcal{H}^{m-1}}\,d\tau.
\end{align}

Finally, by virtue of Lemma \ref{L2.3} and Young's inequality, we obtain that
\begin{align}
&\Big{|}\int^t_0\int_{\Omega}\mathcal{Z}^\alpha\nabla( nc)\cdot\nabla\mathcal{Z}^\alpha\Delta c\,dxd\tau\Big{|}\nonumber\\
\leq&\,\delta\int^t_0\|\nabla \Delta c\|^2_{\mathcal{H}^{m-1}}\,d\tau+D_\delta P(M(t))\int^t_0\big{(}\|n\|^2_{\mathcal{H}^{m-1}}\nonumber\\
&+\|\nabla n\|^2_{\mathcal{H}^{m-1}}+\|c\|^2_{\mathcal{H}^{m-1}}+\|\nabla c\|^2_{\mathcal{H}^{m-1}}\big{)}\,d\tau,\label{3.118}\\
&\Big{|}\int^t_0\int_{\Omega}\mathcal{Z}^\alpha\nabla(u\cdot\nabla c)\cdot\nabla\mathcal{Z}^\alpha\Delta c\,dxd\tau\Big{|}\nonumber\\
\leq&\,\delta\int^t_0\|\nabla \Delta c\|^2_{\mathcal{H}^{m-1}}\,d\tau+D_\delta P(M(t))\int^t_0\big{(}\|\nabla c\|^2_{\mathcal{H}^{m-1}}\nonumber\\
&+\|\nabla^2 c\|^2_{\mathcal{H}^{m-1}}+\|\nabla u\|^2_{\mathcal{H}^{m-1}}+\|u\|^2_{\mathcal{H}^{m-1}}\big{)}\,d\tau.\label{3.119}
\end{align}

Consequently, based on \eqref{3.69}, \eqref{3.115}-\eqref{3.119} and the inductive assumption, we can complete the proof of Lemma \ref{L3.9}.
\end{proof}

It follows from Lemma \ref{L3.4} to Lemma \ref{L3.9}, where $\delta>0$ is suitably small, that
\begin{align}\label{3.119.1}
&\sup_{0\leq\tau\leq t}(\|(n,c,u)\|^2_{\mathcal{H}^m}+\|\nabla n\|^2_{\mathcal{H}^{m-1}}+\|\nabla c\|^2_{\mathcal{H}^{m}}+\|\Delta(n,c)\|^2_{\mathcal{H}^{m-1}})
\nonumber\\
&+\epsilon\int_0^t\|\nabla u\|^2_{\mathcal{H}^{m}}\,d\tau+\int_0^t(\|\nabla n\|^2_{\mathcal{H}^{m}}+\|\Delta c\|^2_{\mathcal{H}^{m}}+\|\nabla\Delta(n,c)\|^2_{\mathcal{H}^{m-1}})\,d\tau\nonumber\\
\leq&\,D_{m+2}\,\Big{\{}\|(n_0,c_0,u_0)\|^2_{\mathcal{H}^m}+\|\nabla n_0\|^2_{\mathcal{H}^{m-1}}+\|\nabla c_0\|^2_{\mathcal{H}^{m}}+\|\Delta (n_0,c_0)\|^2_{\mathcal{H}^{m-1}}\nonumber\\
&+\int_0^t\big{(}\|\nabla^2p_1\|_{\mathcal{H}^{m-1}}\|u\|_{\mathcal{H}^{m}}
+\epsilon^{-1}(\|\nabla p_2\|_{\mathcal{H}^{m-1}}^2+\|p_2\|_{\mathcal{H}^{m-1}}^2)\big{)}d\tau\nonumber\\
&+\big{(}1+P(M(t))\big{)}\int_0^t\big{(}\|(n,c,u)\|_{\mathcal{H}^m}^2+\|\nabla(n,u)\|_{\mathcal{H}^{m-1}}^2+\|\nabla c\|_{\mathcal{H}^{m}}^2\nonumber\\
&+\|\Delta(n,c)\|_{\mathcal{H}^{m-1}}^2\big{)}d\tau+\int_0^t\|n\nabla\phi\|_{\mathcal{H}^{m}}^2d\tau\Big{\}}.
\end{align}

\subsection{Normal Derivative Estimates}
In this subsection, we provide the estimate for $\|\nabla u\|_{\mathcal{H}^{m-1}}$. Note that
$$\|\chi\partial_{y^i}u\|_{\mathcal{H}^{m-1}}\leq D\,\|u\|_{\mathcal{H}^{m}}\quad \text{for}\quad i=1,2,$$
it suffices to estimate $\|\chi\partial_\nu u\|_{\mathcal{H}^{m-1}}$, where $\chi$ is compactly supported in one of the $\Omega_i$ and with value one in a vicinity of the boundary. We shall thus use the local coordinates \eqref{1.18}.
Due to \eqref{3.17}, we immediately obtain that
\begin{align}\label{3.120}
\|\chi\partial_\nu u\cdot \nu\|_{\mathcal{H}^{m-1}}\leq D_m\,\|u\|_{\mathcal{H}^{m}}.
\end{align}
Thus, it remains to estimate $\|\chi\Pi\partial_\nu u\|_{\mathcal{H}^{m-1}}$. We define
\begin{align*}
\eta:=\chi\Pi((\nabla u+(\nabla u)^t)\nu)+2\zeta\chi\Pi u.
\end{align*}
In view of the boundary condition \eqref{1.2}, we have
\begin{align*}
\eta=0\quad \text{on}\quad\partial\Omega.
\end{align*}
Moreover, since $\eta$ have another form in the vicinity of the boundary $\partial\Omega$:
\begin{align}
\eta=\chi\Pi\partial_\nu u+\chi\Pi(\nabla(u\cdot \nu)-\nabla \nu\cdot u-u\times(\nabla\times \nu)+2\zeta u)\label{3.121},
\end{align}
we easily get that
\begin{align*}
\|\chi\Pi\partial_\nu u\|_{\mathcal{H}^{m-1}}\leq&\, D_{m+1}\,(\|\eta\|_{\mathcal{H}^{m-1}}+\|u\|_{\mathcal{H}^{m}}+\|\partial_\nu u\cdot \nu\|_{\mathcal{H}^{m}})\nonumber\\
\leq&\, D_{m+1}\,(\|\eta\|_{\mathcal{H}^{m-1}}+\|u\|_{\mathcal{H}^{m}}).
\end{align*}
Hence, it suffices to estimate $\|\eta\|_{\mathcal{H}^{m-1}}$. We have the following estimates for $\eta$.
\begin{Lemma}\label{L3.10} For every $m\geq1$, we have
\begin{align}\label{3.122}
&\sup_{0\leq\tau\leq t}\|\eta\|^2_{\mathcal{H}^{m-1}}+\epsilon\int^t_0\|\nabla\eta\|^2_{\mathcal{H}^{m-1}}d\tau\nonumber\\
\leq&\, D_{m+2}\Big{\{}\|u_0\|^2_{\mathcal{H}^{m}}+\|\nabla u_0\|^2_{\mathcal{H}^{m-1}}+\epsilon^{-1}\int_0^t\|\nabla p_2\|^2_{\mathcal{H}^{m-1}}d\tau+\epsilon^2\int^t_0\|\nabla u\|_{\mathcal{H}^{m}}^2\,d\tau\nonumber\\
&+\int_0^t(\|\nabla^2p_1\|_{\mathcal{H}^{m-1}}+\|\nabla p_1\|_{\mathcal{H}^{m-1}})\|\eta\|_{\mathcal{H}^{m-1}}d\tau+\int^t_0\|\nabla^2 p_1\|_{\mathcal{H}^{m-1}}\|u\|_{\mathcal{H}^{m}}d\tau\nonumber\\
&+\int^t_0\|\nabla (n\nabla\phi)\|_{\mathcal{H}^{m-1}}\|\eta\|_{\mathcal{H}^{m-1}}d\tau+\int^t_0\|n\nabla\phi\|_{\mathcal{H}^{m-1}}\|u\|_{\mathcal{H}^{m}}d\tau
\nonumber\\
&+\big{(}1+P(M(t))\big{)}\int^t_0(\|u\|_{\mathcal{H}^{m}}^2+\|\nabla u\|^2_{\mathcal{H}^{m-1}}
)\,d\tau\Big{\}}.
\end{align}
\end{Lemma}

\begin{proof}
Setting $\mathcal{N}=\nabla u$, we get from \eqref{1.1.3} that
\begin{align*}
\mathcal{N}_t-\epsilon\Delta \mathcal{N}+ u\cdot\nabla\mathcal{N}=-\mathcal{N}^2-\nabla^2p+\nabla(n\nabla\phi).
\end{align*}
Hence, $\eta$ solves the equation
\begin{align}
\eta_t-\epsilon\Delta\eta+u\cdot\nabla\eta=F-2\chi\Pi(\nabla^2p \nu),\label{3.123}
\end{align}
where $F:=F^b+F^\chi+F^\kappa$ with
\begin{align*}
F^b:=&-\chi\Pi\big{(}((\nabla u)^2+((\nabla u)^t)^2)\nu\big{)}-2\zeta\chi\Pi\nabla p-\chi\Pi\big{(}(\nabla(n\nabla \phi))^2\nonumber\\
&\,-((\nabla(n\nabla \phi))^t)^2)\nu\big{)},\\
F^\chi:=&-\epsilon\Delta\chi(\Pi Su \nu+2\zeta\Pi u)-2\epsilon\nabla\chi\cdot\nabla(\Pi Su\nu+2\zeta\Pi u)
+u^\epsilon\cdot\nabla\chi\Pi(Su\nu+2\zeta u),\\
F^\kappa:=&\chi(u\cdot\nabla\Pi)(Su\nu+2\zeta u)+\chi\Pi(Su ( u\cdot\nabla)\nu)-\epsilon\chi(\Delta\Pi)(Su\nu+2\zeta u)\\
&-2\epsilon\chi\nabla\Pi\cdot\nabla(Su\nu+2\zeta u)-\epsilon\chi\Pi(Su\Delta \nu+2\nabla Su\cdot\nabla \nu).
\end{align*}

Let us start with the case of $m=1$. Due to \eqref{1.1.4}, we get
\begin{align}\label{3.124}
\frac{1}{2}\frac{d}{dt}\|\eta\|^2+\epsilon\int_\Omega|\nabla\eta|^2\,dx=\int_\Omega F\cdot\eta\,dx-\int_\Omega\chi\Pi(\nabla^2p\nu)\cdot\eta\,dx.
\end{align}
Now we estimate the right-hand side terms of \eqref{3.124}. In view of Lemma \ref{L2.3}, we easily get
\begin{align}
\int^t_0\|F^b\|_{\mathcal{H}^{m-1}}^2d\tau\leq \,&D_m\,P(M(t))\int^t_0\|\nabla u\|_{\mathcal{H}^{m-1}}^2d\tau+D\int^t_0\|\nabla (n\nabla\phi)\|_{\mathcal{H}^{m-1}}^2d\tau\nonumber\\
&+D_m\,\int^t_0\|\nabla p\|_{\mathcal{H}^{m-1}}^2\,d\tau,\label{3.125}\\
\int^t_0\|F^\chi\|_{\mathcal{H}^{m-1}}^2d\tau\leq \,&D_{m+1}\,\big{(}1+P(M(t))\big{)}\int^t_0\|u\|_{\mathcal{H}^{m}}^2d\tau+D\,\epsilon^2\int^t_0\|\nabla u\|_{\mathcal{H}^{m}}^2\,d\tau,\label{3.126}\\
\int^t_0\|F^\kappa\|_{\mathcal{H}^{m-1}}^2d\tau\leq \,&D_{m+2}\,\big{(}1+P(M(t))\big{)}\int^t_0(\|u\|_{\mathcal{H}^{m}}^2+\|\nabla u\|_{\mathcal{H}^{m-1}}^2
)\,d\tau\nonumber\\
&+D\,\epsilon^2\int^t_0(\|\nabla u\|_{\mathcal{H}^{m-1}}^2+\|\chi\nabla^2 u\|_{\mathcal{H}^{m-1}}^2)\,d\tau.\label{3.127}
\end{align}
Next, we estimate the term involving the pressure $p$ in \eqref{3.124}. By recalling that $p=p_1+p_2$, we get
\begin{equation}\label{3.128}
\Big{|}\int_0^t\int_\Omega\chi\Pi(\nabla^2p \nu)\cdot\eta \,dxd\tau\Big{|}\leq D \int_0^t\|\nabla^2p_1\|\|\eta\|d\tau+\big{|}\int_0^t\int_\Omega\chi\Pi(\nabla^2p_2 \nu)\cdot\eta\,dxd\tau\big{|}.
\end{equation}
Since $\eta=0$ on the boundary, we can integrate by parts the last term in \eqref{3.128} to obtain
\begin{equation}\label{3.129}
\Big{|}\int_0^t\int_\Omega\chi\Pi(\nabla^2p_2 \nu)\cdot\eta\,dxd\tau\Big{|}\leq D\int_0^t\|\nabla p_2\|(\|\nabla\eta\|+\|\eta\|)d\tau.
\end{equation}
Therefore, putting \eqref{3.125}-\eqref{3.129} into \eqref{3.124} and using Young's inequality , we obtain that
\begin{align}\label{3.130}
&\sup_{0\leq\tau\leq t}\|\eta\|^2+\epsilon\int^t_0\|\nabla\eta\|^2d\tau\nonumber\\
\leq&\, D_3\Big{\{}\|\eta_0\|^2+\epsilon^{-1}\int_0^t\|\nabla p_2\|^2d\tau+\int_0^t(\|\nabla^2p_1\|+\|\nabla p\|+\|\nabla(n\nabla\phi)\|)\|\eta\|d\tau\nonumber\\
&+\delta\epsilon^2\int^t_0\|\chi\nabla^2 u\|^2d\tau+\big{(}1+P(M(t))\big{)}\int^t_0(\|u\|_{\mathcal{H}^{1}}^2+\|\nabla u\|^2)d\tau\Big{\}}.
\end{align}
Due to \eqref{3.120} and \eqref{3.121}, we get
\begin{align}\label{3.131}
\epsilon\|\chi\nabla^2u\|_{\mathcal{H}^{m-1}}&\leq D_{m+2}\,\epsilon(\|\nabla\eta\|_{\mathcal{H}^{m-1}}+\|\nabla u\|_{\mathcal{H}^{m}}+\| u\|_{\mathcal{H}^{m}}).
\end{align}
Furthermore, we have
\begin{align}\label{3.132}
&\sup_{0\leq\tau\leq t}\|\eta\|^2+\epsilon\int^t_0\|\nabla\eta\|^2d\tau\nonumber\\
\leq&\, D_3\Big{\{}\|\eta_0\|^2+\epsilon^{-1}\int_0^t\|\nabla p_2\|^2d\tau+\int_0^t(\|\nabla^2p_1\|+\|\nabla p\|+\|\nabla(n\nabla\phi)\|)\|\eta\|d\tau\nonumber\\
&+\delta\epsilon^2\int^t_0\|\nabla u\|_{\mathcal{H}^{1}}^2d\tau+\big{(}1+P(M(t))\big{)}\int^t_0(\|u\|_{\mathcal{H}^{1}}^2+\|\nabla u\|^2)d\tau\Big{\}}.
\end{align}
Hence, \eqref{3.122} holds for $m=1$.

Now we assume that Lemma \ref{L3.10} is true for $|\alpha|\leq m-2$ and let us consider the situation of $|\alpha|= m-1$. By applying $\mathcal{Z}^\alpha$ to \eqref{3.123}, we have
\begin{align}
\mathcal{Z}^\alpha\eta_t-\epsilon\mathcal{Z}^\alpha\Delta\eta+u\cdot\nabla Z^\alpha\eta
=\mathcal{Z}^\alpha F-\mathcal{Z}^\alpha(2\chi\Pi(\nabla^2p\nu))+\mathcal{C}_4,\label{3.133}
\end{align}
where
\begin{equation}\nonumber
\begin{split}
\mathcal{C}_4:=-[\mathcal{Z}^\alpha, u\cdot\nabla]\eta.
\end{split}
\end{equation}
Multiplying \eqref{3.133} by $\mathcal{Z}^\alpha\eta^\epsilon$, we obtain that
\begin{align}\label{3.134}
\frac{1}{2}\frac{d}{dt}\|\mathcal{Z}^\alpha\eta\|^2=&
\epsilon \int_\Omega\mathcal{Z}^\alpha\Delta\eta\cdot\mathcal{Z}^\alpha\eta\,dx
-2\int_\Omega\mathcal{Z}^\alpha(\chi\Pi(\nabla^2p\nu))\cdot\mathcal{Z}^\alpha\eta\,dx\nonumber\\
&+\int_\Omega\mathcal{Z}^\alpha F\cdot\mathcal{Z}^\alpha\eta\,dx+\int_\Omega\mathcal{C}_4\cdot\mathcal{Z}^\alpha\eta\,dx.
\end{align}

First, let us estimate the viscous term. We observe that
\begin{align}\label{3.135}
&\epsilon\int^t_0\int_\Omega \mathcal{Z}^\alpha\partial_{ii}\eta\cdot \mathcal{Z}^\alpha\eta dxd\tau\nonumber\\
=&-\epsilon\int^t_0\int_\Omega|\partial_i\mathcal{Z}^\alpha\eta|^2 dxd\tau-\epsilon\int^t_0\int_\Omega[\mathcal{Z}^\alpha,\partial_i]\eta\cdot\partial_i\mathcal{Z}^\alpha\eta dxd\tau\nonumber\\
&+\epsilon\int^t_0\int_\Omega[\mathcal{Z}^\alpha,\partial_i]\partial_i\eta\cdot \mathcal{Z}^\alpha\eta dxd\tau,
\end{align}
where $i=1, 2, 3$. In order to estimate the last two terms in the right-hand side of \eqref{3.135}, we use the structure of the commutator $[\mathcal{Z}^\alpha,\partial_i]$ and the expansion
$\partial_i=\beta^1\partial_{y^1}+\beta^2\partial_{y^2}+\beta^3\partial_{y^3}$ in the local basis. We have the following expansion
\begin{align}
[\mathcal{Z}^\alpha,\partial_i]\eta=\sum_{\gamma,|\gamma|\leq|\alpha|-1}d_\gamma\partial_z\mathcal{Z}^\gamma\eta+\sum_{\beta,|\beta|\leq|\alpha|}d_\beta \mathcal{Z}^\beta\eta.\nonumber
\end{align}
Thus, we have
\begin{align}
&\epsilon\Big{|}\int^t_0\int_\Omega[\mathcal{Z}^\alpha,\partial_i]\eta\cdot\partial_i\mathcal{Z}^\alpha\eta\,dxd\tau\Big{|}\nonumber\\
\leq &\,D_m\,\epsilon\int^t_0\|\nabla \mathcal{Z}^{m-1}\eta\|(\|\nabla\eta\|_{\mathcal{H}^{m-2}}+\|\eta\|_{\mathcal{H}^{m-1}})\,d\tau,\label{3.136}\\
&\epsilon\Big{|}\int^t_0\int_\Omega[\mathcal{Z}^\alpha,\partial_i]\partial_i\eta\cdot \mathcal{Z}^\alpha\eta \,dxd\tau\Big{|}\nonumber\\
\leq&\, D_m\,\epsilon\Big{\{}\int^t_0\|\nabla\eta\|_{\mathcal{H}^{m-1}}\|\eta\|_{\mathcal{H}^{m-1}}\,d\tau
+\sum_{|\gamma|\leq m-2}\Big{|}\int^t_0\int_\Omega D_\gamma\partial_z\mathcal{Z}^\gamma\partial_i\eta\cdot\mathcal{Z}^\alpha\eta \,dxd\tau\Big{|}\Big{\}}.\label{3.137}
\end{align}
Furthermore, by virtue of $\mathcal{Z}^\alpha\eta^\epsilon=0$ on $\partial\Omega$ and integration by parts, we obtain that
\begin{align}
&\epsilon\Big{|}\int^t_0\int_\Omega[\mathcal{Z}^\alpha,\partial_i]\partial_i\eta\cdot \mathcal{Z}^\alpha\eta dxd\tau\Big{|}\nonumber\\
\leq \,&D_{m+1}\,\epsilon\int^t_0\|\nabla\eta\|_{\mathcal{H}^{m-1}}(\|\eta\|_{\mathcal{H}^{m-1}}+\|\nabla\eta\|_{\mathcal{H}^{m-2}})\,d\tau.\label{3.138}
\end{align}

Second, we deal with the commutator term $\mathcal{C}_4$. Note that
\begin{align}
\mathcal{C}_4=&-\sum_{|\beta|\geq1,\beta+\gamma=\alpha}\sum_{i=1}^2D_{\beta,\gamma}\mathcal{Z}^\beta u_i\cdot\mathcal{Z}^\gamma\partial_{y^i}\eta\nonumber\\
&-\sum_{|\beta|\geq1,\beta+\gamma=\alpha}D_{\beta,\gamma}\mathcal{Z}^\beta( u\cdot N)\mathcal{Z}^\gamma\partial_z\eta-u\cdot N\sum_{|\beta|\leq m-2}D_\beta\partial_z\mathcal{Z}^\beta\eta.\label{3.139}
\end{align}
By using Lemma \ref{L2.3}, we can easily obtain that
\begin{align}\label{3.140}
&\sum_{|\beta|\geq1,\beta+\gamma=\alpha}\sum_{i=1}^2D_{\beta,\gamma}\int^t_0\|\mathcal{Z}^\beta u_i\cdot\mathcal{Z}^\gamma\partial_{y^i}\eta\|^2\,d\tau\nonumber\\
\leq&\,D_{m+2}\,P(M(t))\int^t_0(\|u^\epsilon\|_{\mathcal{H}^{m}}^2+\|\nabla u^\epsilon\|_{\mathcal{H}^{m-1}}^2) \,d\tau.
\end{align}
Since we want to get an estimate independent of $\partial_z\eta^\epsilon$, by using Hardy's inequality, we have
\begin{align}\label{3.141}
\sum_{|\beta|\leq m-2}\int^t_0\|u\cdot ND_\beta\partial_z\mathcal{Z}^\beta\eta\|^2d\tau\leq&\,\sum_{|\beta|\leq m-2}\int^t_0\|\frac{u^\epsilon\cdot N}{\varphi(z)}D_\beta Z_3\mathcal{Z}^\beta\eta\|^2\,d\tau\nonumber\\
\leq&\,D_{m+2}\,P(M(t))\int^t_0(\|u\|_{\mathcal{H}^{m}}^2+\|\nabla u\|_{\mathcal{H}^{m-1}}^2) \,d\tau.
\end{align}
Also, we note that for $|\beta|\geq1$, $\beta+\gamma=\alpha$ and $|\alpha|=m-1$, it holds
\begin{align}\label{3.142}
\mathcal{Z}^\beta(u\cdot N)\mathcal{Z}^\gamma\partial_z\eta&=\frac{1}{\varphi(z)}\mathcal{Z}^\beta(u\cdot N)\cdot\varphi(z)\mathcal{Z}^\gamma\partial_z\eta\nonumber\\
&=\sum_{|\tilde{\beta}|\leq\beta,|\tilde{\gamma}|\leq\gamma}D_{\tilde{\beta},\tilde{\gamma}}\mathcal{Z}^{\tilde{\beta}}(\frac{u^\epsilon\cdot N}{\varphi(z)})\mathcal{Z}^{\tilde{\gamma}}(Z_3\eta),
\end{align}
where $|\tilde{\beta}|+|\tilde{\gamma}|\leq m-1, |\tilde{\gamma}|\leq m-2$ and $D_{\tilde{\beta},\tilde{\gamma}}$ are some smooth bounded coefficient. By using Hardy's inequality, we have
\begin{align}\label{3.143}
&\sum_{|\beta|\geq1,\beta+\gamma=\alpha}\int^t_0\|D_{\beta,\gamma}\mathcal{Z}^\beta( u\cdot N)\mathcal{Z}^\gamma\partial_z\eta\|^2 \,d\tau\nonumber\\
\leq&\,D_{m+2}\,P(M(t))\int^t_0(\|u\|_{\mathcal{H}^{m}}^2+\|\nabla u\|_{\mathcal{H}^{m-1}}^2) \,d\tau.
\end{align}
Therefore, from \eqref{3.140}-\eqref{3.143}, we get
\begin{align}\label{3.144}
\int^t_0\|\mathcal{C}_4\|^2d\tau\leq&\,D_{m+2}\,P(M(t))\int^t_0(\|u\|_{\mathcal{H}^{m}}^2+\|\nabla u\|_{\mathcal{H}^{m-1}}^2) d\tau.
\end{align}

Next, it remains to deal with the term involving the pressure $p$. As above, we use the split $p=p_1+p_2$ and integrate by parts the term involving $p_2$. We have
\begin{align}\label{3.145}
&\Big{|}\int^t_0\int_\Omega\mathcal{Z}^\alpha(\chi\Pi(\nabla^2p \nu))\cdot\mathcal{Z}^\alpha\eta \,dxd\tau\Big{|}\nonumber\\
\leq & \,D_{m+2}\,\int^t_0\big{(}\|\nabla^2p_1\|_{\mathcal{H}^{m-1}}\|\eta\|_{\mathcal{H}^{m-1}}
+\|\nabla p_2\|_{\mathcal{H}^{m-1}}(\|\nabla\eta\|_{\mathcal{H}^{m-1}}+\|\eta\|_{\mathcal{H}^{m-1}})\big{)}\,d\tau.
\end{align}

Finally, from \eqref{3.125}-\eqref{3.127} and \eqref{3.131}, we get
\begin{align}\label{3.146}
\int^t_0\|F\|_{\mathcal{H}^{m-1}}^2d\tau
\leq&\, D_{m+2}\big{(}1+P(M(t))\big{)}\int^t_0(\|u\|_{\mathcal{H}^{m}}^2+\|\nabla u\|_{\mathcal{H}^{m-1}}^2)\,d\tau\nonumber\\
&+D\epsilon^2\int^t_0(\|\nabla u\|_{\mathcal{H}^{m}}^2+\|\nabla\eta\|_{\mathcal{H}^{m-1}}^2)\,d\tau+D_{m+2}\,\int^t_0\|\nabla p\|_{\mathcal{H}^{m-1}}^2\,d\tau\nonumber\\
&+D\int^t_0\|\nabla (n\nabla\phi)\|_{\mathcal{H}^{m-1}}^2\,d\tau.
\end{align}

By collecting \eqref{3.135}, \eqref{3.136}, \eqref{3.138}, \eqref{3.144}-\eqref{3.146}, Young's inequality, and the inductive assumption, we can get \eqref{3.122}. Hence, the proof of Lemma \ref{L3.10} is completed.
\end{proof}

\subsection{Pressure Estimates} To enclose our a priori estimates, it remains to estimate the pressure terms and the $L^\infty$ norms. The aim of this subsection is to give the pressure estimates and present the $L^\infty$ estimates in next subsection.
\begin{Lemma}\label{L3.11} For every $m\geq2$, we have the following estimates:
\begin{align}
\int_0^t(\|\nabla p_1\|_{\mathcal{H}^{m-1}}^2+\|\nabla^2 p_1\|_{\mathcal{H}^{m-1}}^2)\,d\tau\leq&\, D_{m+2}\,P(M(t))\int_0^t(\|u\|_{\mathcal{H}^{m}}^2+\|\nabla u\|_{\mathcal{H}^{m-1}}^2)\, d\tau\nonumber\\
&+\int_0^t(\|n\nabla \phi\|_{\mathcal{H}^{m}}^2+\|\nabla (n\nabla \phi)\|_{\mathcal{H}^{m-1}}^2)\,d\tau,\label{3.147}\\
\int_0^t(\|p_2\|_{\mathcal{H}^{m-1}}^2+\|\nabla p_2\|_{\mathcal{H}^{m-1}}^2)\,d\tau\leq&\, D_{m+2}\epsilon\int_0^t(\|u\|_{\mathcal{H}^{m}}^2+\|\nabla u\|_{\mathcal{H}^{m-1}}^2)\,d\tau.\label{3.148}
\end{align}
\end{Lemma}
\begin{proof}
From \eqref{3.8.11} and \eqref{3.9}, we obtain that
\begin{equation}
\left\{
\begin{aligned}
\Delta(\partial_t^{\alpha_0} p_1)&=-\partial_t^{\alpha_0}\nabla\cdot(u\cdot\nabla u)-\partial_t^{\alpha_0}\nabla\cdot(n\nabla \phi)\quad \text{in}\quad \Omega,\\
\partial_\nu(\partial_t^{\alpha_0}p_1)&=-\partial_t^{\alpha_0}(u\cdot\nabla u)\cdot \nu-\partial_t^{\alpha_0}(n\nabla \phi)\cdot \nu\quad \text{on}\quad \partial\Omega
\end{aligned}
\right.
\end{equation}
and
\begin{equation}
\left\{
\begin{aligned}
&\Delta (\partial_t^{\alpha_0}p_2)=0\quad \text{in}\quad \Omega,\\
&\partial_\nu(\partial_t^{\alpha_0}p_2)=\epsilon\,\partial_t^{\alpha_0}\Delta u\cdot \nu\quad \text{on}\quad \partial\Omega.
\end{aligned}
\right.
\end{equation}

First, we deal with $p_1$. From the standard elliptic regularity results with Neumann boundary condition, we obtain that
\begin{align*}
&\|\nabla \partial_t^{\alpha_0} p^\epsilon_1\|_{H^{|\alpha_1|}_{co}}+\|\nabla^2 \partial_t^{\alpha_0} p^\epsilon_1\|_{H^{|\alpha_1|}_{co}}\nonumber\\
\leq&\, D_{m+1}\,\big{(}\|\partial_t^{\alpha_0}\nabla\cdot(u\cdot\nabla u+n\nabla\phi)\|_{H^{|\alpha_1|}_{co}}+\|\partial_t^{\alpha_0}(u\cdot\nabla u+n\nabla\phi)\|\nonumber\\
&+|\partial_t^{\alpha_0}(u\cdot\nabla u+n\nabla\phi)\cdot\nu|_{H^{m-|\alpha_0|-\frac{1}{2}}(\partial\Omega)}\big{)},
\end{align*}
where $|\alpha_0|+|\alpha_1|=m-1$. Due to $u\cdot \nu=0$ on $\partial\Omega$ and Lemma \ref{L2.2}, we get
\begin{align*}
&|\partial_t^{\alpha_0}(u\cdot\nabla u+n\nabla\phi)\cdot \nu|_{H^{m-|\alpha_0|-\frac{1}{2}}(\partial\Omega)}\nonumber\\
\leq &\,D_{m+2}\,(\|\nabla (u\otimes u)\|_{\mathcal{H}^{m-1}}+\| u\otimes u\|_{\mathcal{H}^{m}}+\|\nabla (n\nabla \phi)\|_{\mathcal{H}^{m-1}}+\|n\nabla \phi\|_{\mathcal{H}^{m}}).
\end{align*}
Using Lemma \ref{L2.3}, we easily get \eqref{3.147}.

It remains to estimate $p_2$. By using the standard elliptic regularity results with Neumann boundary condition again, we obtain that
\begin{align}
\|\partial_t^{\alpha_0}p_2\|_{H^{m-|\alpha_0|-1}_{co}}+\|\nabla\partial_t^{\alpha_0} p_2\|_{H^{m-|\alpha_0|-1}_{co}}\leq D_m\,\epsilon\, |\partial_t^{\alpha_0}\Delta u^\epsilon\cdot \nu|_{H^{m-|\alpha_0|-\frac{3}{2}}(\partial\Omega)}.\nonumber
\end{align}
Since
\begin{align}
\partial_t^{\alpha_0}\Delta u\cdot \nu=2\Big{(}\partial_t^{\alpha_0}\nabla\cdot(Su \nu)-\sum_j\partial_t^{\alpha_0}(Su\partial_j\nu)_j\Big{)},\nonumber
\end{align}
we have
\begin{align*}
|\partial_t^{\alpha_0}\Delta u\cdot \nu|_{H^{m-|\alpha_0|-\frac{3}{2}}(\partial\Omega)}\leq D\,|\partial_t^{\alpha_0}\nabla\cdot(Su \nu)|_{H^{m-|\alpha_0|-\frac{3}{2}}(\partial\Omega)}\\
+\,D_{m+1}\,|\partial_t^{\alpha_0}\nabla u|_{H^{m-|\alpha_0|-\frac{3}{2}}(\partial\Omega)}.
\end{align*}
Due to \eqref{3.19}, we can further arrive at
\begin{align}
|\partial_t^{\alpha_0}\Delta u\cdot \nu|_{H^{m-|\alpha_0|-\frac{3}{2}}(\partial\Omega)}\leq&\, D\,|\partial_t^{\alpha_0}\nabla\cdot(Su \nu)|_{H^{m-|\alpha_0|-\frac{3}{2}}(\partial\Omega)}\nonumber\\
&+\,D_{m+1}\,|\partial_t^{\alpha_0}u|_{H^{m-|\alpha_0|-\frac{1}{2}}(\partial\Omega)}.\nonumber
\end{align}

Let us estimate $|\partial_t^{\alpha_0}\nabla\cdot(Su \nu)|_{H^{m-|\alpha_0|-\frac{3}{2}}(\partial\Omega)}$. Thank to \eqref{3.17}, we get
\begin{align*}
&|\partial_t^{\alpha_0}\nabla\cdot(Su \nu)|_{H^{m-|\alpha_0|-\frac{3}{2}}(\partial\Omega)}\nonumber\\
\leq& \,D\,|\partial_\nu\partial_t^{\alpha_0}(Su \nu)\cdot \nu|_{H^{m-|\alpha_0|-\frac{3}{2}}(\partial\Omega)}+D\,\big{(}|\Pi\partial_t^{\alpha_0}(Su \nu)|_{H^{m-|\alpha_0|-\frac{1}{2}}(\partial\Omega)}\nonumber\\
&+|\nabla\partial_t^{\alpha_0} u|_{H^{m-|\alpha_0|-\frac{3}{2}}(\partial\Omega)}\big{)}.
\end{align*}
Also, due to \eqref{3.19} and the Navier boundary condition \eqref{1.2}, we get
\begin{align}\label{3.149}
|\partial_t^{\alpha_0}\nabla\cdot(Su \nu)|_{H^{m-|\alpha_0|-\frac{3}{2}}(\partial\Omega)}\leq &\,D\,|\partial_\nu\partial_t^{\alpha_0}(Su \nu)\cdot \nu|_{H^{m-|\alpha_0|-\frac{3}{2}}(\partial\Omega)}\nonumber\\
&+|\partial_t^{\alpha_0}u|_{H^{m-|\alpha_0|-\frac{1}{2}}(\partial\Omega)}.
\end{align}
The first term in the right-hand side of \eqref{3.149} can be estimated as
\begin{align*}
&|\partial_\nu\partial_t^{\alpha_0}(Su \nu)\cdot \nu|_{H^{m-|\alpha_0|-\frac{3}{2}}(\partial\Omega)}\nonumber\\
\leq&\, D\,|\partial_\nu\partial_t^{\alpha_0}(\partial_\nu u\cdot \nu)|_{H^{m-|\alpha_0|-\frac{3}{2}}(\partial\Omega)}+D_{m+1}\,|\nabla\partial_t^{\alpha_0} u|_{H^{m-|\alpha_0|-\frac{3}{2}}(\partial\Omega)}\nonumber\\
\leq&\, D\,|\partial_\nu\partial_t^{\alpha_0}(\partial_\nu u\cdot \nu)|_{H^{m-|\alpha_0|-\frac{3}{2}}(\partial\Omega)}+D_{m+1}\,|\partial_t^{\alpha_0}u|_{H^{m-|\alpha_0|-\frac{1}{2}}(\partial\Omega)}.
\end{align*}
By taking the normal derivative of \eqref{3.17} and using \eqref{1.12}, we obtain that
\begin{align*}
&|\partial_\nu\partial_t^{\alpha_0}(\partial_\nu u\cdot \nu)|_{H^{m-|\alpha_0|-\frac{3}{2}}(\partial\Omega)}\nonumber\\
\leq&\,D\,|\Pi\partial_t^{\alpha_0}\partial_\nu u|_{H^{m-|\alpha_0|-\frac{1}{2}}(\partial\Omega)}+D_{m+1} \,|\nabla\partial_t^{\alpha_0} u|_{H^{m-|\alpha_0|-\frac{3}{2}}(\partial\Omega)}\nonumber\\
\leq &\,D_{m+2}\,|\partial_t^{\alpha_0}u|_{H^{m-|\alpha_0|-\frac{1}{2}}(\partial\Omega)}.
\end{align*}
Consequently, we have
\begin{align}
|\partial_t^{\alpha_0}\Delta u\cdot \nu|_{H^{m-|\alpha_0|-\frac{3}{2}}(\partial\Omega)}\leq D_{m+2}\,|\partial_t^{\alpha_0}u|_{H^{m-|\alpha_0|-\frac{1}{2}}(\partial\Omega)}.\nonumber
\end{align}
By virtue of Lemma \ref{L2.2}, we finally get \eqref{3.148}. Therefore, we complete the proof of Lemma \ref{L3.11}.
\end{proof}

Substituting \eqref{3.122} and \eqref{3.147}-\eqref{3.148} into \eqref{3.119.1}, we can obtain that
\begin{align}\label{3.150}
&\sup_{0\leq\tau\leq t}\big{(}\|(n,c,u)\|^2_{\mathcal{H}^m}+\|\nabla(n,u)\|^2_{\mathcal{H}^{m-1}}+\|\nabla c\|^2_{\mathcal{H}^{m}}
\nonumber\\
&+\|\Delta(n,c)\|^2_{\mathcal{H}^{m-1}}\big{)}+\epsilon\int_0^t(\|\nabla u\|^2_{\mathcal{H}^{m}}+\|\nabla^2 u\|^2_{\mathcal{H}^{m-1}})\,d\tau\nonumber\\
&+\int_0^t(\|\nabla n\|^2_{\mathcal{H}^{m}}+\|\Delta c\|^2_{\mathcal{H}^{m}}+\|\nabla\Delta(n,c)\|^2_{\mathcal{H}^{m-1}})\,d\tau\nonumber\\
\leq&\,D_{m+2}\,\Big{\{}N_m(0)
+\big{(}1+P(M(t))\big{)}\int_0^tN_m(\tau)d\tau\nonumber\\
&+\int_0^t(\|n\nabla\phi\|_{\mathcal{H}^{m}}^2+\|\nabla(n\nabla\phi)\|_{\mathcal{H}^{m-1}}^2)\,d\tau\Big{\}}.
\end{align}

\subsection{$L^\infty$ estimates}
In order to close the estimate \eqref{3.150}, we need to give the $L^\infty$ estimates for $(n,c,u)$.
\begin{Lemma}\label{L3.12}
For every $m\geq 4$, we have the following estimates:
\begin{align}
\|n\|_{W^{2,\infty}}^2\leq&\, D (N_m(t)+N_m^2(t)+N_m^3(t)),\label{3.151}\\
\|u\|_{\mathcal{H}^{2,\infty}}^2\leq&\, D N_m(t),\label{3.152}\\
\|c\|_{W^{2,\infty}}^2\leq&\, D (N_m(t)+N_m^2(t)),\label{3.153}\\
\|\nabla c\|_{\mathcal{H}^{1,\infty}}^2\leq&\, DN_m(t),\label{3.154}\\
\|\nabla\Delta c\|_{L^\infty}^2\leq&\,  D (N_m(t)+N_m^2(t)).\label{3.155}
\end{align}
\end{Lemma}
\begin{proof}
In view of Lemma \ref{L2.2} and \eqref{3.69}, we obtain that
\begin{align}
\|n\|_{W^{1,\infty}}^2+\|\nabla n\|_{\mathcal{H}^{1,\infty}}^2+\|u\|_{\mathcal{H}^{2,\infty}}^2+\|\nabla c\|_{\mathcal{H}^{1,\infty}}^2+\|c\|_{\mathcal{H}^{1,\infty}}^2\leq D N_m(t).\label{3.156}
\end{align}
Hence, we get \eqref{3.152} and \eqref{3.154}. Also, by virtue of Lemma \ref{L2.2} and \eqref{3.69}, we have
\begin{align}
\|\nabla^2c\|^2_{L^{\infty}}\leq\,& D(\|\Delta c\|^2_{L^{\infty}}+\|\nabla c\|^2_{H^{1,\infty}_{co}})\nonumber\\
\leq\,& D(\|\nabla\Delta c\|^2_{H^1_{co}}+\|\Delta c\|^2_{H^2_{co}}+\|\nabla c\|^2_{H^3_{co}}).\label{3.157}
\end{align}
Due to the equations \eqref{1.1.2} and \eqref{3.156}, we deduce that
\begin{align}
\|\nabla\Delta c\|^2_{H^1_{co}}\leq D (N_m(t)+N_m^2(t)).\label{3.158}
\end{align}
The combination of \eqref{3.156}-\eqref{3.158} yields \eqref{3.153}.
Based on the equation \eqref{1.1.2}, \eqref{3.153} and \eqref{3.156}, we can easily get
\begin{align}
\|\nabla\Delta c\|_{L^\infty}^2\leq D\|\nabla(c_t+u\cdot\nabla c+nc)\|_{L^\infty}^2\leq D (N_m(t)+N_m^2(t)),\nonumber
\end{align}
which gives \eqref{3.155}. Finally, similar to \eqref{3.157}, we have
\begin{align}
\|\nabla^2n\|^2_{L^{\infty}}\leq\,& D(\|\nabla\Delta n\|^2_{H^1_{co}}+\|\Delta n\|^2_{H^2_{co}}+\|\nabla n\|^2_{H^3_{co}}).\label{3.159}
\end{align}
By virtue of the equation \eqref{1.1.1} and \eqref{3.152}-\eqref{3.156}, we obtain that
\begin{align}
\|\nabla\Delta n\|^2_{H^1_{co}}\leq D(N_m(t)+N_m^2(t)+N_m(t)\|\nabla\Delta c\|^2_{H^1_{co}}).\label{3.160}
\end{align}
Furthermore, with the help of the equation \eqref{1.1.2}, we arrive at
\begin{align}
\|\nabla\Delta c\|^2_{H^1_{co}}\leq\,& D(\|\nabla c\|_{\mathcal{H}^2}^2+\|\nabla(u\cdot\nabla c)\|^2_{H^1_{co}}+\|\nabla(nc)\|^2_{H^1_{co}})\nonumber\\
\leq\,& D(N_m(t)+N_m^2(t)).\label{3.161}
\end{align}
From \eqref{3.156} and \eqref{3.159}-\eqref{3.161}, we get \eqref{3.151}.
\end{proof}

Finally, we prove the estimate for $\|\nabla u\|_{\mathcal{H}^{1,\infty}}$.
\begin{Lemma}\label{L3.13} For $m\geq6$, we have the following estimate:
\begin{align}
\|\nabla u\|_{\mathcal{H}^{1,\infty}}^2\leq&\, D\big{(}N_m(0)+N_m(t)+(1+P(N_m(t)))\nonumber\\
&\times\int^t_0P(N_m(\tau))d\tau\big{)}+\delta\epsilon\int^t_0\|\nabla^2u\|^2_{\mathcal{H}^{4}}\,d\tau,\label{3.162}
\end{align}
where $\delta$ is a small enough constant.
\end{Lemma}
\begin{proof}
We observe that, away from the boundary, the following estimate holds:
\begin{equation}
\|\beta_i\nabla u^\epsilon\|_{L^\infty}^2+\|\beta_i \mathcal{Z}\nabla u^\epsilon\|_{L^\infty}^2\leq D\,\|u^\epsilon\|_{\mathcal{H}^m},\quad m\geq 4,\nonumber
\end{equation}
where $\{\beta_i\}$ is a partition of unity subordinated to the covering \eqref{c}. In order to estimate the near boundary parts, we adopt the ideas in the Proposition $21$ of \cite{MR}. Here, we use a local parametrization in the vicinity of the boundary given by a normal geodesic system:
$$ \Psi^\nu(y,z)=
\left(
\begin{array}{c}
y\\
\psi(y)\\
\end{array}
\right)
-z\nu(y),$$
where
$$\nu(y)=\frac{1}{\sqrt{1+|\nabla\psi(y)|^2}}
\left(
\begin{array}{c}
\partial_1\psi(y)\\
 \partial_2\psi(y)\\
 -1
\end{array}
\right).$$
Now, we can extend $\nu$ and $\Pi$ in the interior by setting
$$\nu(\Psi^\nu(y,z))=\nu(y),\quad\Pi(\Psi^\nu(y,z))=\Pi(y).$$
We observe
$\partial_z=\partial_\nu$ and
$$\left(
\begin{array}{c}
\partial_{y^i}
\end{array}
\right)\Big{|}_{\Psi^\nu(y,z)}\cdot
\left(
\begin{array}{c}
\partial_z
\end{array}
\right)\Big{|}_{\Psi^\nu(y,z)}=0.\nonumber
$$
Hence, the Riemann metric $g$ has the following form
\begin{equation}\nonumber
g(y,z)=
\left(
\begin{matrix}
\widetilde{g}(y,z)&0\\
0&1\\
\end{matrix}
\right).
\end{equation}
Consequently, the Laplacian in this coordinate system reads:
\begin{equation}\nonumber
\Delta f=\partial_{zz}f+\frac{1}{2}\partial_z(\ln|g|)\partial_zf+\Delta_{\widetilde{g}}f,
\end{equation}
where $|g|$ is the determinant of the matrix $g$ and $\Delta_{\widetilde{g}}$ is defined by
\begin{align}\label{3.163}
\Delta_{\widetilde{g}}f=\frac{1}{|\widetilde{g}|^{\frac{1}{2}}}\sum_{1\leq i,j\leq2}\partial_{y^i}(\widetilde{g}^{ij}|\widetilde{g}|^{\frac{1}{2}}\partial_{y^j}f).
\end{align}
Here, $\{\widetilde{g}^{ij}\}$ is the inverse matrix to $g$ and \eqref{3.163} only involves the tangential derivatives.

With these preparation, we now turn to estimate the near boundary parts. By using Lemma \ref{3.12} and \eqref{3.17}, we have
\begin{align}
\|\chi\nabla u\|_{\mathcal{H}^{1,\infty}}\leq \,D_3\,(\|\chi \Pi\partial_\nu u\|_{\mathcal{H}^{1,\infty}}+\|u\|_{\mathcal{H}^{m}}+\|\nabla u\|_{\mathcal{H}^{m-1}}).\label{3.164}
\end{align}
Hence, we need to estimate $\|\chi \Pi\partial_\nu u\|_{\mathcal{H}^{1,\infty}}$. To this end, we first introduce the vorticity
$$\omega=\nabla\times u.$$
We find that
\begin{align}\label{3.165}
\Pi(\omega\times \nu)&=\Pi(\nabla u-(\nabla u)^t)\nu\nonumber\\
&=\Pi(\partial_\nu u-\nabla(u\cdot \nu)+(\nabla \nu)^tu+u\times(\nabla\times \nu)).
\end{align}
Consequently, we have
\begin{equation}\label{3.166}
\|\chi \Pi\partial_\nu u\|_{\mathcal{H}^{1,\infty}}\leq D_3\,(\|\chi\Pi(\omega\times \nu)\|_{\mathcal{H}^{1,\infty}}+\|u\|_{\mathcal{H}^{2,\infty}}).
\end{equation}
By using \eqref{3.164} again, we get
\begin{equation}\label{3.167}
\|\chi\nabla u\|_{\mathcal{H}^{1,\infty}}\leq\, D_3\,(\|\chi\Pi(\omega\times \nu)\|_{\mathcal{H}^{1,\infty}}+\|u\|_{\mathcal{H}^{m}}+\|\nabla u\|_{\mathcal{H}^{m-1}}).
\end{equation}

In order to conclude the estimate \eqref{3.162}, we only need to estimate $\|\chi\Pi(\omega\times \nu)\|_{\mathcal{H}^{1,\infty}}$. By setting in the support of $\chi$
\begin{align}
&\widetilde{\omega}(y,z):=\omega^\epsilon(\Psi^\nu(y,z)),\quad\widetilde{u}(y,z):=u(\Psi^\nu(y,z)),\nonumber
\end{align}
we have
\begin{align*}
\widetilde{\omega}_t+(\widetilde{u})^1\partial_{y^1}\widetilde{\omega}
+(\widetilde{u})^2\partial_{y^2}\widetilde{\omega}
+\widetilde{u}\cdot n\partial_z
\widetilde{\omega}
=&\epsilon(\partial_{zz}\widetilde{\omega}+\frac{1}{2}\partial_z(\ln|g|)\partial_z\widetilde{\omega}+\Delta_{\widetilde{g}}\widetilde{\omega})+\widetilde{F}_1,\\
\widetilde{u}_t+(\widetilde{u})^1\partial_{y^1}\widetilde{u}
+(\widetilde{u})^2\partial_{y^2}\widetilde{u}+\widetilde{u}\cdot n\partial_z
\widetilde{u}=&\epsilon(\partial_{zz}\widetilde{u}+\frac{1}{2}\partial_z(\ln|g|)\partial_z\widetilde{u}+\Delta_{\widetilde{g}}\widetilde{u})\\
&-(\nabla p)\circ\Psi^\nu-(n\nabla\phi)\circ\Psi^\nu,
\end{align*}
where
\begin{align*}
\widetilde{F}_1:=F_1(\Psi^\nu(y,z)),\quad F_1:=(\omega\cdot\nabla)u-\nabla n\times\nabla\phi.
\end{align*}

By using \eqref{1.12} and \eqref{3.165} on the boundary, we have
$$\Pi(\widetilde{\omega}\times \nu)=2\Pi((\nabla \nu)^t\widetilde{u}-\zeta\widetilde{u}).$$
Consequently, we introduce the following quantity:
\begin{align}
&\widetilde{\eta}(y,z):=\chi\Pi(\widetilde{\omega}\times \nu-2(\nabla \nu)^t\widetilde{u}+2\zeta\widetilde{u}).\label{3.168.1}
\end{align}
We thus get that $\widetilde{\eta}(y,0)=0$ and that $\widetilde{\eta}$ solves the equation
\begin{align}
\!\!\!\widetilde{\eta}_t+(\widetilde{u})^1\partial_{y^1}\widetilde{\eta}
+(\widetilde{u})^2\partial_{y^2}\widetilde{\eta}
+\widetilde{u}\cdot \nu\partial_z
\widetilde{\eta}=&\epsilon(\partial_{zz}\widetilde{\eta}+\frac{1}{2}\partial_z(\ln|g|)\partial_z\widetilde{\eta})\nonumber\\
&+\chi\Pi(\widetilde{F}_1\times \nu)+\widetilde{F}^u+\widetilde{F}^\chi+\widetilde{F}^\kappa,\label{3.168}
\end{align}
where
\begin{align*}
\widetilde{F}^u=&\,2\chi\Pi((\nabla \nu)^t(\nabla p+n\nabla\phi)-\zeta\nabla p-\zeta(n\nabla\phi))\circ\Psi^\nu,\\
\widetilde{F}^\chi=&\,(((\widetilde{u})^1\partial_{y^1}+(\widetilde{u})^2\partial_{y^2}+\widetilde{u}\cdot \nu\partial_z)\chi)\Pi(\widetilde{\omega}\times \nu-2(\nabla \nu)^t\widetilde{u}+2\zeta\widetilde{u})\\
&-\epsilon(\partial_{zz}\chi+2\partial_z\chi\partial_z+\frac{1}{2}\partial_z(\ln|g|)\partial_z\chi)\Pi(\widetilde{\omega}\times \nu-2(\nabla \nu)^t\widetilde{u}+2\zeta\widetilde{u}),\\
\widetilde{F}^\kappa=&\chi(((\widetilde{u})^1\partial_{y^1}+(\widetilde{u})^2\partial_{y^2})\Pi)(\widetilde{\omega}\times \nu
-2(\nabla \nu)^t\widetilde{u}+2\zeta\widetilde{u})
+\epsilon\chi\Pi(\Delta_{\widetilde{g}}\widetilde{\omega}\times \nu)\\
&-2\epsilon\chi\Pi((\nabla \nu)^t\Delta_{\widetilde{g}}\widetilde{u})-2\chi\Pi((((\widetilde{u})^1\partial_{y^1}+(\widetilde{u})^2\partial_{y^2})(\nabla \nu)^t)\widetilde{u})\\
&+\chi\Pi(\widetilde{\omega}\times((\widetilde{u})^1\partial_{y^1}+(\widetilde{u})^2\partial_{y^2})\nu)
-2\zeta\epsilon\chi\Pi(\Delta_{\widetilde{g}}\widetilde{u}).
\end{align*}
We know that both $\Pi$ and $\nu$ do not dependent the normal variable. Due to $\Delta_{\widetilde{g}}$ only involving the tangential derivatives and the derivatives of $\chi$ compactly supported away from the boundary, we easily obtain that
\begin{align}
\|\chi\Pi(\widetilde{F}_1\times \nu)\|_{\mathcal{H}^{1,\infty}}\leq &\,D\,(\nabla u\|_{\mathcal{H}^{1,\infty}}^2+\|\nabla n \times \nabla\phi\|_{\mathcal{H}^{1,\infty}}),\label{3.169}\\
\|\widetilde{F}^u\|_{\mathcal{H}^{1,\infty}}\leq& \,D_3\,(\|\Pi \nabla p\|_{\mathcal{H}^{1,\infty}}+\|\Pi (n\nabla \phi)\|_{\mathcal{H}^{1,\infty}}),\label{3.170}\\
\|\widetilde{F}^\chi\|_{\mathcal{H}^{1,\infty}}\leq &\,D_3\,(\|u\|_{\mathcal{H}^{1,\infty}}\|u\|_{\mathcal{H}^{2,\infty}}+\epsilon\|u\|_{\mathcal{H}^{3,\infty}}),\label{3.171}\\
\|\widetilde{F}^\kappa\|_{\mathcal{H}^{1,\infty}}\leq&\, D_4\,(\|u\|_{\mathcal{H}^{1,\infty}}^2+\| \nabla u\|_{\mathcal{H}^{1,\infty}}\| u\|_{\mathcal{H}^{1,\infty}}\nonumber\\
&+\epsilon\|u\|_{\mathcal{H}^{3,\infty}}+\epsilon\|\nabla u\|_{\mathcal{H}^{3,\infty}}).\label{3.172}
\end{align}
Therefore, by using Lemmas \ref{L2.2} and \ref{L3.12}, we get that
\begin{align}\label{3.173}
\|\widetilde{F}\|_{\mathcal{H}^{1,\infty}}^2\leq \,&D_4\,\big{(}\|\Pi \nabla p\|_{\mathcal{H}^{1,\infty}}^2+\epsilon^2\|\nabla u\|_{\mathcal{H}^{3,\infty}}^2+N_m(t)+N_m(t)^2\big{)},
\end{align}
where $\widetilde{F}:=\chi\Pi(\widetilde{F}_1\times \nu)+\widetilde{F}^u+\widetilde{F}^\chi+\widetilde{F}^\kappa$.
A crucial estimate towards the proof of Lemma \ref{L3.13} is the following:
\begin{Lemma}[\!\!\cite{MR}]\label{L3.14} Consider $\rho$ a smooth solution of
\begin{align}
\rho_t+u\cdot\nabla\rho=\epsilon\partial_{zz}\rho+\mathcal{S},\quad z>0,\quad\rho(t,y,0)=0\nonumber
\end{align}
for some smooth divergence free vector field $u$ such that $u\cdot\nu$ vanishes on the boundary. Assume that $\rho$ and $\mathcal{S}$ are compactly supported in $z$. Then, we have the estimate
\begin{align*}
\|\rho\|_{\mathcal{H}^{1,\infty}}\leq \,&D\|\rho(0)\|_{\mathcal{H}^{1,\infty}}+D\int_0^t\big{\{}(\|u\|_{\mathcal{H}^{2,\infty}}+\|\partial_zu\|_{\mathcal{H}^{1,\infty}})\nonumber\\
&\times(\|\rho\|_{\mathcal{H}^{1,\infty}}
+\|\rho\|_{\mathcal{H}^{m_0+3}})+\|\mathcal{S}\|_{\mathcal{H}^{1,\infty}}\big{\}\,d\tau}
\end{align*}
for $m_0\geq2$.
\end{Lemma}
In order to use Lemma \ref{L3.14}, we shall eliminate $\partial_z(\ln|g|)\partial_z\widetilde{\eta}$ in \eqref{3.168}. We set $$\widetilde{\eta}:=\frac{1}{|g|^{\frac{1}{4}}}\overline{\eta}=\overline{\gamma}\,\overline{\eta}.$$
We note that
\begin{align}\label{3.174}
\|\widetilde{\eta}\|_{\mathcal{H}^{1,\infty}}\sim\|\overline{\eta}\|_{\mathcal{H}^{1,\infty}},
\end{align}
and $\overline{\eta}$ solve the equations
\begin{align}
\!\!\!\overline{\eta}_t+(\widetilde{u})^1\partial_{y^1}\overline{\eta}
+(\widetilde{u})^2\partial_{y^2}\overline{\eta}
+(u\cdot n)\partial_z\overline{\eta}
-\epsilon\partial_{zz}\overline{\eta}
=\mathcal{S},\label{3.175}
\end{align}
where
\begin{align}
\mathcal{S}:=\frac{1}{\overline{\gamma}}\big{(}\chi\Pi(\widetilde{F}_1\times\nu)+\widetilde{F}^u+\widetilde{F}^\chi+\widetilde{F}^\kappa+\epsilon\partial_{zz}\overline{\gamma}\,\overline{\eta}
+\frac{\epsilon}{2}\partial_z\ln|g|\partial_z\overline{\gamma}\,\overline{\eta}-(\widetilde{u}\cdot\nabla\overline{\gamma})\overline{\eta}\big{)}.
\end{align}
Applying Lemma \ref{L3.14} to \eqref{3.175}, we obtain that
\begin{align}
\|\overline{\eta}\|_{\mathcal{H}^{1,\infty}}\leq \,D\Big{\{}\|\overline{\eta}_0\|_{\mathcal{H}^{1,\infty}}
+\int_0^t\big{(}(\|u\|_{\mathcal{H}^{2,\infty}}+\|\partial_zu\|_{\mathcal{H}^{1,\infty}})\nonumber\\
\times(\|\overline{\eta}\|_{\mathcal{H}^{1,\infty}}
+\|\overline{\eta}\|_{\mathcal{H}^{m_0+3}})+\|\mathcal{S}\|_{\mathcal{H}^{1,\infty}}\big{)}d\tau\Big{\}}.\label{3.177}
\end{align}
It remains to estimate $\|\mathcal{S}\|_{\mathcal{H}^{1,\infty}}$. Due to  Lemmas \ref{L2.2}, \ref{3.11} and \eqref{3.173}, we have
\begin{align}
\int^t_0\|\mathcal{S}\|^2_{\mathcal{H}^{1,\infty}}\,d\tau
&\leq D\int^t_0\big{(}\|\Pi \nabla p^\epsilon\|^2_{\mathcal{H}^{1,\infty}}+\epsilon^2\|\nabla u^\epsilon\|^2_{\mathcal{H}^{3,\infty}}+N_m(t)+N_m(t)^2\big{)}d\tau\nonumber\\
&\leq D_{m+2}\big{(}1+P(N_m(t))\big{)}\int^t_0P(N_m(\tau))d\tau+\delta\epsilon\int^t_0\|\nabla^2u\|^2_{\mathcal{H}^{4}}\,d\tau.\label{3.178}
\end{align}
By virtue of \eqref{3.178} and Lemma \ref{3.13}, we get
\begin{align}
\|\overline{\eta}\|_{\mathcal{H}^{1,\infty}}^2\leq& D_{m+2}\Big{\{}\|\overline{\eta}_0\|_{\mathcal{H}^{1,\infty}}^2
+(1+P(N_m(t)))\int^t_0P(N_m(\tau))d\tau\nonumber\\
&+\delta\epsilon\int^t_0\|\nabla^2u\|^2_{\mathcal{H}^{4}}\,d\tau\Big{\}}.\label{3.179}
\end{align}
Finally, the combination of \eqref{3.164}, \eqref{3.167}, \eqref{3.168.1} and \eqref{3.179} yields \eqref{3.162}.
Therefore, we complete the proof of Lemma \ref{L3.14}.
\end{proof}

\subsection{Proof of Theorem \ref{Th3.1}}
It suffices to combine \eqref{3.150}, Lemma \ref{L3.12}, and Lemma \ref{L3.13}.

\section{Proof of Theorem \ref{Th1}}\label{Sec4}
In this section, we will show how to combine our a priori estimates to prove the uniform existence results. Let us fix $m\geq6$ and consider the initial data satisfy
\begin{align}\label{4.1}
\mathcal{I}_m(0)=\sup_{0<\epsilon\leq1}\|(n_0^\epsilon,c_0^\epsilon,u_0^\epsilon)\|_{\mathcal{E}^{m,\epsilon}_{CNS}}\leq \widetilde{D}_3.
\end{align}
For such initial data, we are not aware of a local existence result for the problem \eqref{1.1.1}-\eqref{1.2}, so we first need to prove the local existence results for \eqref{1.1.1}-\eqref{1.2} by using the energy estimates obtained in Section \ref{Sec3} and a classical iteration scheme. By virtue of the definition of $\mathcal{E}^{m,\epsilon}_{CNS}$, there exists a sequence of smooth approximate initial data $(n_0^{\epsilon,\delta},c_0^{\epsilon,\delta},u_0^{\epsilon,\delta})$ ($\delta$ being a regularization parameter) which has enough space regularity so that the time derivatives at initial data can be defined by the chemotaxis-Navier-Stokes system and the boundary compatibility conditions can be satisfied.

Fixed $\epsilon\in(0,1]$, we construct approximate solutions as follows:\\
(1) Define $(n^0,c^0,u^0)=(n_0^{\epsilon,\delta},c_0^{\epsilon,\delta},u_0^{\epsilon,\delta})$.\\
(2) Assume that $(n^{k-1},c^{k-1},u^{k-1})$ has been defined for $k\geq1$. Let $(n^{k},c^{k},u^{k})$ be the unique solution to the following linearized initial boundary value problem:
\begin{equation}\label{4.2}
\left\{
\begin{aligned}
&n^k_t-\Delta n^k=-u^{k-1}\cdot\nabla n^{k}-\nabla\cdot( n^{k}\cdot\nabla c^{k-1}),\\
&c^k_t-\Delta c^k=-u^{k-1}\cdot\nabla c^{k}- n^{k-1} c^{k},\\
&u^k_t-\epsilon\Delta u^k+u^{k-1}\cdot\nabla u^{k}+\nabla p^{k}=n^k\nabla\phi,\\
&\dv u^k=0,\\
&(n^{k},c^{k},u^{k})|_{t=0}=(n_0^{\epsilon,\delta},c_0^{\epsilon,\delta},u_0^{\epsilon,\delta}).
\end{aligned}
\right.
\end{equation}
in $(0,T)\times\Omega$ with the boundary conditions
\begin{align}\label{4.3}
\frac{\partial n^k}{\partial\nu}=\frac{\partial c^k}{\partial\nu}=0,\quad u^k\cdot \nu=0,\quad (Su^k\cdot \nu)_\tau=-\zeta u^k_\tau.
\end{align}
Since $n^k$, $c^k$ and $u^k$ are decoupled, the existence of the global smooth solution $(n^{k},c^{k},u^{k})$ of \eqref{4.2} and \eqref{4.3} can be obtained by using the classical methods, for example, see \cite{GTY,YHH,SAL}.

By using the a priori estimates given in Theorem \ref{Th3.1} and an induction argument, we obtain that there exist a uniform time $\widetilde{T}_1$ and constant $\widetilde{D}_4$ (independent of $\epsilon$ and $\delta$) such that it holds for $(n^{k},c^{k},u^{k}), k\geq1$ that
\begin{align}\label{4.4}
&\sup_{0\leq\tau\leq t}\big{\{}\|(n^k,c^k,u^k)\|_{\mathcal{H}^m}^2+\|\nabla(n^k,u^k)\|_{\mathcal{H}^{m-1}}^2+\|\nabla c^k\|_{\mathcal{H}^{m}}^2\nonumber\\
&+\|\Delta(n^k,c^k)\|_{\mathcal{H}^{m-1}}^2+\|\nabla u^k\|^2_{1,\infty}\big{\}}
+\epsilon\int_0^t(\|\nabla u^k\|^2_{\mathcal{H}^{m}}+\|\nabla^2 u^k\|^2_{\mathcal{H}^{m-1}})\,d\tau\nonumber\\
&+\int_0^t(\|\nabla n^k\|^2_{\mathcal{H}^{m}}+\|\Delta c^k\|^2_{\mathcal{H}^{m}}
+\|\nabla\Delta(n^k,c^k)\|^2_{\mathcal{H}^{m-1}})\,d\tau\leq\widetilde{D}_4,\quad \forall\, t\in[0,\widetilde{T}_1],
\end{align}
where $\widetilde{T}_1$ and $\widetilde{D}_4$ depend only on $\mathcal{I}_m(0)$. In view of the above uniform estimates, there exists a uniform time $\widetilde{T}_2$ (independent $\epsilon$ and $\delta$) such that $(n^{k},c^{k},u^{k})$ converges to a limit $(n^{\epsilon,\delta},c^{\epsilon,\delta},u^{\epsilon,\delta})$ as $k\rightarrow+\infty$ in the following strong sense:
\begin{align*}
(n^{k},c^{k})\rightarrow(n^{\epsilon,\delta},c^{\epsilon,\delta})\quad& \text{in}\quad L^\infty(0,\widetilde{T}_2;H^1),\\
u^k\rightarrow u^{\epsilon,\delta}\quad &\text{in}\quad L^\infty(0,\widetilde{T}_2;L^2), \\
\nabla u^k\rightarrow \nabla u^{\epsilon,\delta}\quad &\text{in}\quad L^2(0,\widetilde{T}_2;L^2).
\end{align*}
It is easy to deduce that $(n^{\epsilon,\delta},c^{\epsilon,\delta},u^{\epsilon,\delta})$ is a weak solution to the system \eqref{1.1.1}-\eqref{1.2} with the initial data $(n_0^{\epsilon,\delta},c_0^{\epsilon,\delta},u_0^{\epsilon,\delta})$. Furthermore, due to the lower semi-continuity of norms, we obtain that
\begin{align}\label{4.5}
&\sup_{0\leq\tau\leq t}(\|(n^{\epsilon,\delta},c^{\epsilon,\delta},u^{\epsilon,\delta})\|_{\mathcal{H}^m}^2+\|\nabla(n^{\epsilon,\delta},u^{\epsilon,\delta})\|_{\mathcal{H}^{m-1}}^2+\|\nabla c^{\epsilon,\delta}\|_{\mathcal{H}^{m}}^2\nonumber\\
&+\|\Delta(n^{\epsilon,\delta},c^{\epsilon,\delta})\|_{\mathcal{H}^{m-1}}^2+\|\nabla u^{\epsilon,\delta}\|^2_{1,\infty})
+\epsilon\int_0^t(\|\nabla u^{\epsilon,\delta}\|^2_{\mathcal{H}^{m}}+\|\nabla^2 u^{\epsilon,\delta}\|^2_{\mathcal{H}^{m-1}})\,d\tau\nonumber\\
&+\int_0^t(\|\nabla n^{\epsilon,\delta}\|^2_{\mathcal{H}^{m}}+\|\Delta c^{\epsilon,\delta}\|^2_{\mathcal{H}^{m}}
+\|\nabla\Delta(n^{\epsilon,\delta},c^{\epsilon,\delta})\|^2_{\mathcal{H}^{m-1}})\,d\tau\leq\widetilde{D}_4,\quad \forall\, t\in[0,\widetilde{T}_2].
\end{align}
Based on the uniform estimate \eqref{4.5} for $(n^{\epsilon,\delta},c^{\epsilon,\delta},u^{\epsilon,\delta})$, we can pass the limit $\delta\rightarrow 0$ to get a strong solution $(n^{\epsilon},c^{\epsilon},u^{\epsilon})$ of the system \eqref{1.1.1}-\eqref{1.2} with initial data $(n_0^{\epsilon},c_0^{\epsilon},u_0^{\epsilon})$ satisfying \eqref{4.1} by using a strong compactness arguments. Indeed, it follows from \eqref{4.5} that $(n^{\epsilon,\delta},c^{\epsilon,\delta},u^{\epsilon,\delta},\nabla c^{\epsilon,\delta})$ is bounded uniformly in $L^\infty(0,\widetilde{T}_2;H^{m}_{co})$ while $(\nabla n^{\epsilon,\delta}, \nabla u^{\epsilon,\delta},\partial_t\nabla c^{\epsilon,\delta})$ is bounded uniformly in $L^\infty(0,\widetilde{T}_2;H^{m-1}_{co})$, $\partial_t\nabla n^{\epsilon,\delta}$ is bounded uniformly in $L^\infty(0,\widetilde{T}_2;H^{m-2}_{co})$,
  $(\Delta n^{\epsilon,\delta},\Delta c^{\epsilon,\delta})$ bounded uniformly in $L^\infty(0,\widetilde{T}_2;H^{m-1}_{co})$,  and $(\partial_tn^{\epsilon,\delta},\partial_tc^{\epsilon,\delta},\partial_tu^{\epsilon,\delta})$ bounded uniformly in $L^\infty(0,\widetilde{T}_2;H^{m-1}_{co})$. Then, we obtain that $(n^{\epsilon,\delta},c^{\epsilon,\delta},u^{\epsilon,\delta},\nabla c^{\epsilon,\delta})$ is compact in $\mathcal{C}(0,\widetilde{T}_2;H^{m-1}_{co})$ and $\nabla n^{\epsilon,\delta}$ compact in $\mathcal{C}(0,\widetilde{T}_2;H^{m-2}_{co})$ by using the strong compactness argument. In particular, there exists a sequence $\delta_k\rightarrow0^+$, $(n^{\epsilon},c^{\epsilon},u^{\epsilon},\nabla c^{\epsilon})\in\mathcal{C}(0,\widetilde{T}_2;H^{m-1}_{co})$ and $\nabla n^{\epsilon}\in\mathcal{C}(0,\widetilde{T}_2;H^{m-2}_{co})$ such that
\begin{align*}
&\nabla n^{\epsilon,\delta_k}\rightarrow\nabla n^{\epsilon} \quad\text{in}\quad\mathcal{C}(0,\widetilde{T}_2;H^{m-2}_{co})\quad \text{as}\quad \delta_k\rightarrow0^+,\\
&(n^{\epsilon,\delta_k},c^{\epsilon,\delta_k},u^{\epsilon,\delta_k},\nabla c^{\epsilon,\delta_k})\rightarrow(n^{\epsilon},c^{\epsilon},u^{\epsilon},\nabla c^{\epsilon}) \quad\text{in}\quad\mathcal{C}(0,\widetilde{T}_2;H^{m-1}_{co})\quad \text{as}\quad \delta_k\rightarrow0^+.
\end{align*}
Moreover, applying the lower semi-continuity of norms to \eqref{4.5}, we obtain the bounds \eqref{4.5} for $(n^{\epsilon},c^{\epsilon},u^{\epsilon})$. In view of \eqref{4.5} and Lemma \ref{L2.2}, we have that
\begin{align}
\sup_{0\leq\tau\leq \widetilde{T}_2}\|(n^{\epsilon,\delta_k},c^{\epsilon,\delta_k},u^{\epsilon,\delta_k},\nabla n^{\epsilon,\delta_k},\nabla c^{\epsilon,\delta_k})-(n^{\epsilon},c^{\epsilon},u^{\epsilon},\nabla n^\epsilon,\nabla c^\epsilon)\|_{L^\infty}\rightarrow0.
\end{align}
Hence, it is easy to deduce that $(n^{\epsilon},c^{\epsilon},u^{\epsilon})$ is a weak solution of the chemotaxis-Navier-Stokes equations. The uniqueness of the solution $(n^{\epsilon},c^{\epsilon},u^{\epsilon})$ comes directly from the Lipschitz regularity of solution. Therefore, the whole family $(n^{\epsilon,\delta},c^{\epsilon,\delta},u^{\epsilon,\delta})$ converges to $(n^{\epsilon},c^{\epsilon},u^{\epsilon})$. Taking $\widetilde{T}_0=\widetilde{T}_2$ and $\widetilde{D}_1=\widetilde{D}_4$, we complete the proof of Theorem \ref{Th1}.

\section{The proof of Theorem \ref{Th2}}\label{Sec5}
We can use the compactness argument that is almost the same as the one needed for the proof of Theorem \ref{Th1} to prove Theorem \ref{Th2}. Hence we omit the details here.

\medskip
{\bf Acknowledgements:} I am very grateful to Professor Fucai Li for his valuable suggestions and encouragement during preparing this paper. This paper is supported by NSFC (Grant No.11271184).


\end{document}